\documentclass[12pt,reqno]{amsart}
\usepackage[T1]{fontenc}
\usepackage[english]{babel}
\usepackage[utf8]{inputenc}
\usepackage[square,numbers]{natbib}
\usepackage{amssymb}
\usepackage{amsmath}
\usepackage{xcolor}
\usepackage{amsfonts}
\usepackage{amsthm}
\usepackage{csquotes}
\usepackage[a4paper,margin=2.5cm]{geometry}
\usepackage{graphicx} 
\usepackage{comment}
\usepackage{tikz}
\usepackage{subcaption}
\usepackage{enumitem}
\usepackage{float}
\usepackage{textgreek}

\usepackage[square,numbers]{natbib}

\usepackage[colorlinks]{hyperref}
\hypersetup{
 colorlinks = true,
 linkcolor = blue,
}
\setcitestyle{numbers}
\theoremstyle{plain}

\newcommand{\1}{\mathbf{1}}
\newcommand{\esp}{\operatorname{\mathbb{E}}}

\newcommand{\var}{\operatorname{var}}

\newcommand{\R}{\mathbb{R}}
\newcommand{\C}{\mathbb{C}}
\newcommand{\N}{\mathbb{N}}

\newcommand{\A}{\mathcal{A}}

\newcommand{\cumuf}{\kappa^{\operatorname{free}}}

\newcommand{\crblocks}{\operatorname{\bf cr}}
\newcommand{\ccblocks}{\operatorname{\bf cc}}
\newcommand{\ncrblocks}{\operatorname{\bf ncr}}

\newcommand{\proj}{\operatorname{Proj}}

\title[Tensor CLT]{Central Limit Theorem for tensor products of free variables}
\author[C.\ Lancien \and P.\ O.\ Santos \and P.\ Youssef]{%
        Cécilia Lancien \and 
        Patrick Oliveira Santos \and 
        Pierre Youssef
        }

\address{C\'ecilia Lancien. CNRS \& Institut Fourier UMR 5582, Universit\'e Grenoble Alpes, Grenoble, France.}
\email{cecilia.lancien@univ-grenoble-alpes.fr}
\address{Patrick Oliveira Santos. Laboratoire d'Analyse et de Math\'ematiques Appliqu\'ees UMR 8050, Universit\'e Gustave Eiffel, Universit\'e Paris Est Cr\'eteil, Marne-la-Vall\'ee, France.}
\email{patrick.oliveirasantos@u-pem.fr}
\address{Pierre Youssef. Division of Science, NYU Abu Dhabi, Abu Dhabi, UAE \& Courant Institute of Mathematical Sciences, New York University, New York, USA.}
\email{yp27@nyu.edu}

\begin{document}
\newtheorem{theorem}{Theorem}[section]

\newtheorem{corollary}[theorem]{Corollary}
\newtheorem{lemma}[theorem]{Lemma}
\newtheorem{conjecture}[theorem]{Conjecture}
\newtheorem{proposition}[theorem]{Proposition}

\theoremstyle{definition}
\newtheorem{example}[theorem]{Example}
\newtheorem{definition}[theorem]{Definition}
\newtheorem{remark}[theorem]{Remark}

\maketitle
\begin{abstract}
We establish a central limit theorem for tensor product random variables $c_k:=a_k \otimes a_k$, where $(a_k)_{k \in \N}$ is a free family of variables. We show that if the variables $a_k$ are centered, the limiting law is the semi-circle. Otherwise, the limiting law depends on the mean and variance of the variables $a_k$ and corresponds to a free interpolation between the semi-circle law and the classical convolution of two semi-circle laws.
\end{abstract}

\section{Introduction}\label{sec: introduction}

The Free Central Limit Theorem serves as a foundational principle in free probability \cite{voiculescu1985symmetries}, \cite[Lecture 8]{nica2006lectures}. It asserts that as the number of freely independent operators summed together approaches infinity, the distribution of the normalized sum tends towards an asymptotically semi-circular shape. This mirrors the classical Central Limit Theorem but with independence conditions replaced by free independence (also known as freeness) and the Gaussian limit substituted with a semi-circular limit. More precisely, let $(\A,\tau)$ be a unital noncommutative probability space equipped with a faithful tracial state $\tau$ \cite[Lecture 1]{nica2006lectures}. We say that subalgebras $\mathcal{A}_1,\ldots, \mathcal{A}_d\subset \A$ are free if 
\begin{align*}
    \tau(a_1\ldots a_p)=0,
\end{align*}
whenever $p \ge 1$, $a_i \in \mathcal{A}_{j_i}$, $\tau(a_i)=0$ for all $i\in[p]$ and $j_1 \ne j_2 \ne \cdots \ne j_p$. We say that random variables $a_1,\ldots,a_d \in \mathcal{A}$ are free if their generated algebras are free. 
We say that a sequence of (self-adjoint) variables $a_n \in (\A_n,\tau_n)$ converges in distribution to a variable $a \in (\A,\tau)$ if
\begin{align*}
    \tau_n(a_n^p) \to \tau(a^p),
\end{align*}
for all integers $p \ge 0$ and we denote it $a_n \Rightarrow a$. We denote $a-\lambda:=a-\lambda \1$, where $\1 \in \A$ is the unit in the algebra. As usual, $\tau(a)$ is the mean of $a$ and the variance is given by
\begin{align*}
    \var(a)=\tau((a-\tau(a))^2).
\end{align*}

The Free Central Limit Theorem states that if $a_1,\ldots,a_n\in(\A,\tau)$ are free self-adjoint i.i.d random variables with mean $\lambda$ and variance $\sigma^2$, then 
$$
\frac{1}{\sigma\sqrt{n}}\sum_{k\in[n]} (a_k-\lambda) \Rightarrow \mu_{sc},
$$
where $\mu_{sc}$ denotes the semi-circle distribution whose density is given by 
\begin{align*}\label{semicircular density}
    f_{sc}(x)=\frac{1}{2\pi}\sqrt{4-x^2} \mathbf{1}_{|x| \le 2}.
\end{align*}

The goal of this paper is to establish a central limit theorem for the tensor product of free random variables. Concretely, given $a_1,\ldots,a_n\in(\A,\tau)$ free self-adjoint i.i.d random variables, we aim at studying the convergence of the normalized sequence
\begin{equation}\label{eq: intro-tensor-setting}
       \frac{1}{\sqrt{n}}\sum_{k \in [n]}\left(a_k \otimes a_k-\tau\otimes \tau(a_k\otimes a_k)\right),
\end{equation}
in the product space $(\A\otimes \A,\tau\otimes \tau)$. 

Just as free probability captures the limiting behavior of random matrices, the above expression appears naturally as the limiting object corresponding to several models of random quantum channels \cite{lancien2023limiting}. Indeed, given $M_1,\ldots,M_n \in \mathcal{M}_d(\C)$ independent random self-adjoint matrices, it was shown in \cite{lancien2023limiting} that the empirical spectral distribution (ESD) of the quantum channel 
\begin{align*}
    \Delta_{d,n}:=\frac{1}{\sqrt{n}}\sum_{k \in [n]} \left(M_k \otimes M_k-\esp [M_k \otimes M_k]\right),
\end{align*}
having the $M_k$'s as random Kraus operators and with fixed Kraus rank $n$, converges as $d\to \infty$ to the expression in \eqref{eq: intro-tensor-setting} with the $a_k$'s being the corresponding limits of the ESD of the $M_k$'s. Moreover, it was in particular shown that if the random matrices $M_k$ are centered, then the ESD of $\Delta_{d,n}$ converges as $n,d\to \infty$ to the semi-circle distribution. These two statements combined suggest that, in the case where the $a_k$'s are centered, an analogue of the free central limit theorem should hold for the $a_k\otimes a_k$'s. Whereas these heuristics indicate that the semi-circle distribution should appear as the limit of \eqref{eq: intro-tensor-setting} when the $a_k$'s are centered, the convergence and the explicit limit are not clear in the general case. 
The goal of this paper is to address this by establishing the convergence of the expression in \eqref{eq: intro-tensor-setting} and identifying the limiting object. 
The latter, as we show, depends on the mean and variance of the variables $a_k$'s and represents a free interpolation between a semi-circle distribution and the classical convolution of two semi-circle distributions.

Random matrix models of the form 
\begin{align} \label{eq:M-QIT}
    M = \sum_{k \in [n]}M_k \otimes M_k,
\end{align}
for $M_1,\ldots,M_n \in \mathcal{M}_d(\C)$ independent random self-adjoint matrices, are in fact useful in other areas of Quantum Information Theory. When the $M_k$'s are positive semidefinite matrices, normalizing $M$ by its trace produces a model for a random separable quantum state. Little is known about the typical asymptotic spectrum of separable states, contrary to that of entangled ones \cite{ambainis2012}. Moreover, a random matrix $M$ of the form \eqref{eq:M-QIT} appears naturally when performing a so-called realignment operation on a quantum state. Understanding the spectrum of the realignment of a state is important as it gives information on the entanglement of the state. In \cite{aubrun2012}, this was done in the particular case where the $M_k$'s are Gaussian matrices (corresponding to the case where the state is a normalized Wishart matrix). The results and techniques we develop here (combined with those in \cite{lancien2023limiting}) could be useful in addressing the questions mentioned above.

Given a measure $\mu$, we denote
\begin{align*}
    (t\mu)(A):=\mu(t^{-1}A),
\end{align*}
its dilation by $t \ne 0$, where $A$ is any Borel set in $\R$. Equivalently, if $a$ is a random variable with distribution $\mu$, then $ta$ has distribution $t\mu$.

The following is our main theorem.
\begin{theorem}\label{th: main theorem}
    Let $a \in (\A,\tau)$ be a self-adjoint random variable with mean $\tau(a)=\lambda$ and variance $\var(a)=\sigma^2\ne 0$. Denote  
    $$
    \delta^2:=\var(a\otimes a)= \sigma^2(\sigma^2+2\lambda^2),
    $$
    and
    \begin{align*}
    q:=\frac{2\lambda^2}{\sigma^2+2\lambda^2} \in [0,1].
\end{align*}
Given $(a_k)_{k\in \N}$ a sequence of free copies of $a$, the normalized sum 
    \begin{align*}
        S_n:= \frac{1}{\delta \sqrt{n}}\sum_{k \in [n]}(a_k\otimes a_k-\lambda^2)
    \end{align*}
    converges in distribution as $n\to \infty$ to 
    \begin{equation}\label{eq: limit measure}
        \mu_q:=\sqrt{q}\left(\frac{1}{\sqrt{2}}\mu_{sc}+\frac{1}{\sqrt{2}}\mu_{sc}\right)\boxplus\sqrt{1-q}\,\mu_{sc},
    \end{equation}
    where $+$ denotes the classical convolution and $\boxplus$ denotes the free convolution.
\end{theorem}

The difficulty in analyzing $S_n$ stems from the complicated dependence structure exhibited by tensors, combining classical independence (between the two legs of the tensor) and freeness (between the variables across tensors). In the centered case, similar computations were made for semi-circle random variables \cite{nica2016meandric,dehornoy2014dual}.
It would be of interest to design a general notion of independence corresponding to the tensor case, analyze its properties, derive the corresponding limit theorems, and characterize the corresponding universal objects. One particular generalization is by replacing the tensor product with the product of $\varepsilon$-independent random variables \cite{speicher2016mixtures,mlotkowski2004lambdafree,speicher2019quantum}. A direct consequence of Theorem~\ref{th: main theorem} is that such a notion cannot, in general, reduce to freeness. 

\begin{corollary}\label{cor: non freeness of tensor products}
    Let $a_1,\ldots,a_n \in (\A,\tau)$ be self-adjoint free i.i.d noncentered random variables. 
Then $\{a_k\otimes a_k: k \in [n]\}$ are not free.
\end{corollary}
The above corollary trivially follows from Theorem~\ref{th: main theorem}, since if the $a_k\otimes a_k$'s were free, the limit of their normalized sum would be the semi-circle distribution contradicting the conclusion of Theorem~\ref{th: main theorem} (when $\lambda\neq 0$). This fact was originally proved in \cite{collins2017freeness} where, more generally, the freeness of tensors of free variables was characterized.

This paper is organized as follows. In Section \ref{sec: prelim}, we recall some definitions and notations. In Section~\ref{sec: properties limit}, we provide some properties of the limiting measure appearing in Theorem~\ref{th: main theorem}. Section~\ref{sec: existence} establishes the existence of the limit, while Section~\ref{section: proof} is dedicated to the proof of Theorem~\ref{th: main theorem}.

\subsection*{Acknowledgments} 
The last named author would like to thank Guillaume C\'ebron and Roland Speicher for helpful discussions. The second named author also thanks Philippe Biane for helpful discussions. Part of this work was initiated during a stay of the second named author at New York University in Abu Dhabi, partly funded by a doctoral mobility grant delivered by Universit\'e Gustave Eiffel; he would like to thank both institutions for their support and the excellent working assumptions. The first named author was supported by the ANR projects ESQuisses (grant number ANR-20-CE47-0014-01), STARS (grant number ANR-20-CE40-0008), and QTraj (grant number ANR-20-CE40-0024-01).

\section{Preliminaries and notations}\label{sec: prelim}



Given $p\in \N$, a partition $\pi=\{V_1, \ldots ,V_k\}$ of $[p]$ is a collection of disjoint sets $V_1,\ldots,V_k$ called blocks such that 
\begin{align*}
    V_1 \cup \cdots \cup V_k=[p].
\end{align*}
We denote by $P(p)$ the set of partitions of $[p]$. We say that a partition $\pi \in P(p)$ is \textit{connected} (also referred to as a linked diagram in \cite{NH1979}) if no proper subinterval of $[p]$ can be written as the union of blocks of $\pi$. A partition $\pi \in P(p)$ has a crossing $i<k<j<l$ if there exist two disjoint blocks $V_1,V_2 \in \pi$ such that $\{i,j\}\subset V_1$ and $\{k,l\} \subset V_2$. In this case, we say that $V_1$ crosses $V_2$.
A block $V\in \pi$ is crossing if there exists another $V' \in \pi$ such that $V'$ crosses $V$, and noncrossing if it does not cross any other block. 
We say that a partition $\pi \in P(p)$ is a \textit{noncrossing} partition if all its blocks are noncrossing. 
We denote by $P^{\operatorname{con}}(p)$ (resp.~$NC(p)$) the set of all connected (resp.~noncrossing) partitions of $[p]$; see Figure \ref{fig: partitions}. We note that the cardinal $|NC(p)|$ is equal to $C_p$, the $p$-th Catalan number.
\begin{figure}[H]
     \centering
     \begin{subfigure}[b]{0.3\textwidth}
         \centering
         \begin{tikzpicture}[scale=0.50]
            \draw[-] (1,0) edge (1,1);
            \draw[-] (1,1) edge (3,1);
            \draw[-] (3,1) edge (3,0);
        
            \draw[-] (2,0) edge (2,2);
            \draw[-] (2,2) edge (5,2);
            \draw[-] (5,2) edge (5,0);
        
            \draw[-] (4,0) edge (4,1);
            \draw[-] (4,1) edge (6,1);
            \draw[-] (6,1) edge (6,0);
            \draw[-] (6,1) edge (7,1);
            \draw[-] (7,1) edge (7,0);
        \end{tikzpicture}
        \caption{Connected partition}
        \label{fig: connected partition}
    \end{subfigure}
     \hfill
     \begin{subfigure}[b]{0.3\textwidth}
         \centering
         \begin{tikzpicture}[scale=0.50]
            \draw[-] (1,0) edge (1,1);
            \draw[-] (1,1) edge (3,1);
            \draw[-] (3,1) edge (3,0);
        
            \draw[-] (2,0) edge (2,2);
            \draw[-] (2,2) edge (4,2);
            \draw[-] (4,2) edge (4,0);
            \draw[-] (4,2) edge (5,2);
            \draw[-] (5,2) edge (5,0);
        
            \draw[-] (6,0) edge (6,1);
            \draw[-] (6,1) edge (7,1);
            \draw[-] (7,1) edge (7,0);
        \end{tikzpicture}
         \caption{General partition}
         \label{fig: general partition}
     \end{subfigure}
     \hfill
     \begin{subfigure}[b]{0.3\textwidth}
         \centering
         \begin{tikzpicture}[scale=0.50]
            \draw[-] (1,0) edge (1,2);
            \draw[-] (1,2) edge (4,2);
            \draw[-] (4,2) edge (4,0);
            \draw[-] (4,2) edge (5,2);
            \draw[-] (5,2) edge (5,0);
        
            \draw[-] (2,0) edge (2,1);
            \draw[-] (2,1) edge (3,1);
            \draw[-] (3,1) edge (3,0);
        
            \draw[-] (6,0) edge (6,1);
        \end{tikzpicture}
         \caption{Noncrossing partition}
         \label{fig: noncrossing partition}
     \end{subfigure}
        \caption{Examples of partitions.}
        \label{fig: partitions}
\end{figure}
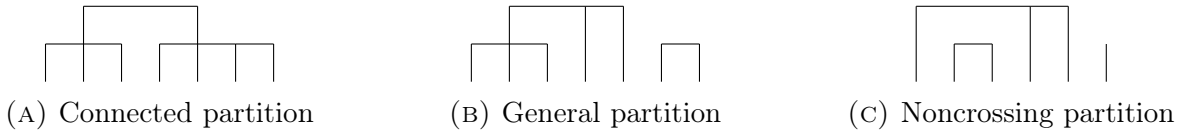
Finally, for a partition $\pi \in P(p)$, we denote by $G(\pi)$ its \textit{intersection graph}. It is the graph over the blocks of $\pi$ such that two blocks are connected if they cross, under some arbitrary labeling. We say that a partition $\pi$ is a bipartite partition if its intersection graph is bipartite and denote it $\pi \in P^{\operatorname{bi}}(p)$; see Figure \ref{fig: partitions gc} for the partitions in Figure \ref{fig: partitions}.
\begin{figure}[H]
     \centering
     \begin{subfigure}[b]{0.3\textwidth}
         \centering
         \begin{tikzpicture}[scale=0.50]
            \node at (0,0) (A) {};
            \fill[black] (A) circle(1.5pt);
            \node at (1,1) (B) {};
            \fill[black] (B) circle(1.5pt);
            \node at (2,0) (C) {};
            \fill[black] (C) circle(1.5pt);

            \draw[-] (0,0) edge (1,1);
            \draw[-] (1,1) edge (2,0);
        \end{tikzpicture}
        \caption{Connected partition}
        \label{fig: connected partition gc}
    \end{subfigure}
     \hfill
     \begin{subfigure}[b]{0.3\textwidth}
         \centering
         \begin{tikzpicture}[scale=0.50]
            \node at (0,0) (A) {};
            \fill[black] (A) circle(1.5pt);
            \node at (1,1) (B) {};
            \fill[black] (B) circle(1.5pt);
            \node at (2,0) (C) {};
            \fill[black] (C) circle(1.5pt);

            \draw[-] (0,0) edge (1,1);
        \end{tikzpicture}
         \caption{General partition}
         \label{fig: general partition gc}
     \end{subfigure}
     \hfill
     \begin{subfigure}[b]{0.3\textwidth}
         \centering
         \begin{tikzpicture}[scale=0.50]
            \node at (0,0) (A) {};
            \fill[black] (A) circle(1.5pt);
            \node at (1,1) (B) {};
            \fill[black] (B) circle(1.5pt);
            \node at (2,0) (C) {};
            \fill[black] (C) circle(1.5pt);
        \end{tikzpicture}
         \caption{Noncrossing partition}
         \label{fig: noncrossing partition gc}
     \end{subfigure}
        \caption{Examples of intersection graphs.}
        \label{fig: partitions gc}
\end{figure}
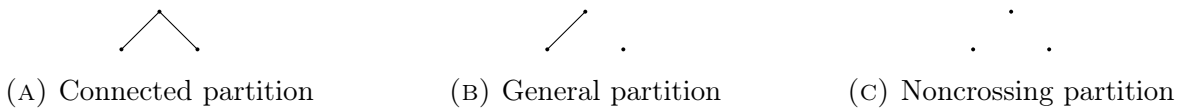
Given $\pi\in P(p)$, we denote $|\pi|$ its number of blocks and $\crblocks(\pi)$ its number of crossing blocks. Therefore, the number of noncrossing blocks of $\pi$ is
\begin{align*}
    \ncrblocks(\pi):=|\pi|-\crblocks(\pi).
\end{align*}

Given $\pi \in P(p)$, we denote $\ccblocks(\pi)$ its number of connected components. We denote $P_2(p), P_2^{\operatorname{bi}}(p)$, $P_2^{\operatorname{con}}(p)$, $P_2^{\operatorname{bicon}}(p)$ and $NC_2(p)$ the set of pair partitions, bipartite pair partitions, connected pair partitions, bipartite connected pair partitions, and noncrossing pair partitions, respectively, that is, those partitions such that all of their blocks have cardinality two.

A pair partition $\pi \in P_2(p)$ can be decomposed into its crossing connected components, namely, let $\hat{\pi} \in P(p)$ be the choice of connected components and, for each block $T \in \hat{\pi}$, draw a connected pair partition $\pi_T \in P_2^{\operatorname{con}}(T)$. By definition, $\hat{\pi} \in NC(p)$ as otherwise two disjoint components would meet ($\hat{\pi}$ is called the noncrossing closure of $\pi$ in \cite{L2002}). The mapping 
\begin{align}\label{equation: mapping pi connected components}
    \Phi:\pi \mapsto (\hat{\pi},(\pi_T)_{T \in \hat{\pi}})
\end{align} 
is a bijection that will be used throughout the proof of Theorem~\ref{th: main theorem}; see Figure \ref{fig: bijection Phi}.
\begin{figure}[H]
    \centering
    \begin{subfigure}[b]{0.4\textwidth}
        \centering
        \begin{tikzpicture}[scale=0.50]
            \draw[-] (1,0) edge (1,1);
            \draw[-] (1,1) edge (3,1);
            \draw[-] (3,1) edge (3,0);

            \draw[-] (2,0) edge (2,2);
            \draw[-] (2,2) edge (4,2);
            \draw[-] (4,2) edge (4,0);

            \draw[-] (5,0) edge (5,2);
            \draw[-] (5,2) edge (8,2);
            \draw[-] (8,2) edge (8,0);

            \draw[-] (6,0) edge (6,1);
            \draw[-] (6,1) edge (7,1);
            \draw[-] (7,1) edge (7,0);
        \end{tikzpicture}
        \caption{$\pi$}
        \label{fig: pi}
    \end{subfigure}
    \begin{subfigure}[b]{0.4\textwidth}
        \centering
        \begin{tikzpicture}[scale=0.50]
            \draw[-] (1,0) edge (1,1);
            \draw[-] (1,1) edge (4,1);
            \draw[-] (3,1) edge (3,0);
            \draw[-] (2,0) edge (2,1);
            \draw[-] (4,1) edge (4,0);

            \draw[-] (5,0) edge (5,2);
            \draw[-] (5,2) edge (8,2);
            \draw[-] (8,2) edge (8,0);

            \draw[-] (6,0) edge (6,1);
            \draw[-] (6,1) edge (7,1);
            \draw[-] (7,1) edge (7,0);
        \end{tikzpicture}
        \caption{$\hat{\pi}$}
        \label{fig: hat pi}
    \end{subfigure}
    \begin{subfigure}[b]{0.4\textwidth}
        \centering
        \begin{tikzpicture}[scale=0.50]
            \draw[-] (1,0) edge (1,1);
            \draw[-] (1,1) edge (3,1);
            \draw[-] (3,1) edge (3,0);

            \draw[-] (2,0) edge (2,2);
            \draw[-] (2,2) edge (4,2);
            \draw[-] (4,2) edge (4,0);
        \end{tikzpicture}
        \caption{$\pi_{\{1,2,3,4\}}$}
        \label{fig: pi_1}
    \end{subfigure}
    \begin{subfigure}[b]{0.4\textwidth}
        \centering
        \begin{tikzpicture}[scale=0.50]
            \draw[-] (1,0) edge (1,1);
            \draw[-] (1,1) edge (2,1);
            \draw[-] (2,1) edge (2,0);
        \end{tikzpicture}
        \caption{$\pi_{\{5,8\}}\cong\pi_{\{6,7\}}$}
        \label{fig: pi_2}
    \end{subfigure}
    \caption{A partition $\pi$ and its image $\Phi(\pi)$.}
    \label{fig: bijection Phi}
\end{figure}
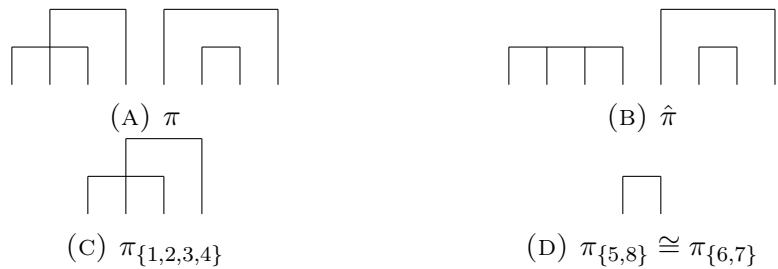
Note, for instance, that $|\hat{\pi}|=\ccblocks(\pi)$. 
We denote
\begin{align*}
    \proj(\hat{\pi}):=\{((\pi_T)_{T \in \hat{\pi}}): \pi_T \in P_2^{\operatorname{con}}(T), \ \forall T \in \hat{\pi}\}.
\end{align*}

Given $a_1,\ldots,a_n \in \A$, we denote by $\cumuf_n(a_1,\ldots,a_n)$ their free cumulants, namely, for any $i \in [n]^k$, we have
\begin{align*}
    &\tau(a_{i_1}\cdots a_{i_k})=\sum_{\pi \in NC(k)}\cumuf_\pi(a_{i_1},\ldots,a_{i_k}),\\
    &\cumuf_\pi(a_{i_1},\ldots,a_{i_k})=\prod_{V=\{v_1<\cdots<v_l\} \in \pi}\cumuf_{|V|}(a_{i_{v_1}},\ldots,a_{i_{v_l}}).
\end{align*}
This is known as the moment-cumulant formula \cite[Notation 11.5]{nica2006lectures}. Here and throughout the paper, we denote $V=\{v_1<\cdots<v_l\}$ a block $V=\{v_1,\ldots,v_l\}$ such that $v_1< \cdots <v_l$; see  \cite[Lecture 11]{nica2006lectures}.
Note also that if the variables $a_1,\ldots,a_n$ are free, the free mixed cumulants vanish \cite[Proposition 11.15]{nica2006lectures}.
We denote $\cumuf_n(a)$ the free cumulants of a random variable $a$. The moment-cumulant formula implies the following, which is going to be used extensively in Subsection \ref{subsection: contribution of connected partitions}; see \cite[Lecture 5, Equation 5.6]{nica2006lectures}.
\begin{lemma}\label{lemma: contribution of ac1ac2}
    Let $a,c_1,c_2$ be variables such that $a$ is free from $\{c_1,c_2\}$. Then
    \begin{align*}
        \tau(ac_1ac_2)=\var(a)\tau(c_1)\tau(c_2)+\tau^2(a)\tau(c_1c_2).
    \end{align*}
\end{lemma}
We will equivalently denote $\cumuf_n(\mu)$ the free cumulants of a random variable $a$ with distribution $\mu$. Given two measures $\mu_a$ and $\mu_b$, the free convolution $\mu_a \boxplus \mu_b$ denotes the distribution of $a+b$, where $a$ and $b$ are free random variables with distribution $\mu_a$ and $\mu_b$, respectively. The classical convolution $\mu_a+\mu_b$ denotes the distribution of $a+b$, where now $a$ and $b$ are classical independent random variables with distribution $\mu_a$ and $\mu_b$, respectively.



\section{Properties of the limiting measure}\label{sec: properties limit}

In this section, we summarize some of the properties of the measure $\mu_q$ appearing in \eqref{eq: limit measure}. 
We start by computing the free cumulants and the moments of $\mu_1$.
\begin{proposition}\label{proposition: cumulants and moments mu_1}
    Let
    \begin{align*}
        \mu_1=\frac{1}{\sqrt{2}}\mu_{sc}+\frac{1}{\sqrt{2}}\mu_{sc}.
    \end{align*}
    Then the following hold.
    \begin{enumerate}
        \item Its odd moments vanish and, for every $p \in \N$, its $2p$-th moment is given by
        \begin{align*}
            \int x^{2p}\, \text{d}\mu_1&=2^{-p}\sum_{l=0}^p\binom{2p}{2l}C_lC_{p-l} =2^{-p}\sum_{\pi \in P_2^{\operatorname{bi}}(2p)}2^{\ccblocks(\pi)},
        \end{align*}
        where we recall that $C_l$ denotes the $l$-th Catalan number.
        \item Its odd free cumulants vanish, $\cumuf_2(\mu_1)=1$ and for any even integer $n \ge 4$, we have
        \begin{align*}
            \cumuf_n(\mu_1)=2\left(\frac{1}{2}\right)^{n/2}|P_2^{\operatorname{bicon}}(n)|.
        \end{align*}
    \end{enumerate}
\end{proposition}
\begin{proof}
    Let $x_1,x_2$ be two classical i.i.d semi-circle random variables. Then
    \begin{align*}
        \int x^{2p}\, \text{d}\mu_1=2^{-p}\tau\left((x_1+x_2)^{2p}\right).
    \end{align*}
    Since the variables are classical independent, and in particular, they commute, we get
    \begin{align*}
        \tau\left((x_1+x_2)^{2p}\right)=2^{-p}\sum_{l=0}^{2p}\binom{2p}{l}\tau(x_1^{l})\tau(x_2^{2p-l}).
    \end{align*}
    The odd moments of $x_1$ vanish and, for any $l \ge 1$, $\tau(x_1^{2l})=C_l$, hence
    \begin{align*}
        \int x^{2p}\, \text{d}\mu_1=2^{-p}\sum_{l=0}^p\binom{2p}{2l}C_lC_{p-l}.
    \end{align*}
    Now consider the sequence $m_p$ given by
    \begin{align*}
        m_p:=2^{-p}\sum_{\pi \in P_2^{\operatorname{bi}}(2p)}2^{\ccblocks(\pi)}.
    \end{align*}
    For every bipartite pair partition $\pi$, we can decompose it into its bipartite sets. Namely, let $\mathcal{V}_1=\{V_1,\ldots, V_m\}$ and $\mathcal{V}_2=\{V_{m+1},\ldots, V_p\}$ be the bipartition of its blocks in two disjoint families of vertices. Then, the blocks $V_1,\ldots,V_m$ are noncrossing from one another, and so are $V_{m+1},\ldots, V_p$. Let
    \begin{align*}
        I=\bigcup_{i=1}^m V_i \subseteq [2p].
    \end{align*}
    Then
    \begin{align*}
        &\pi_1:=\{V_1,\ldots,V_m\} \in NC_2\left(I\right);\\
        &\pi_2:=\{V_{m+1},\ldots,V_p\} \in NC_2\left(I^c\right).
    \end{align*}
    We denote $(\pi_1,\pi_2,I,I^c)$ a left-right noncrossing representation of $\pi$. We note that if $(\pi_1,\pi_2,I,I^c)$ is a left-right noncrossing representation of $\pi$, so is $(\pi_2,\pi_1,I^c,I)$. Let $R(\pi)$ be the set of all left-right noncrossing representations of $\pi$. We note that $|R(\pi)|$ corresponds to the number of ways the vertices of $G(\pi)$ can be split into two independent families of vertices. Therefore, it is clear that
    \begin{align*}
        |R(\pi)|=2^{\ccblocks(\pi)}.
    \end{align*}
    Hence
    \begin{align}\label{eq: sum over bipartite partitions and representations}
        \sum_{\pi \in P_2^{\operatorname{bi}}(2p)}2^{\ccblocks(\pi)}=\sum_{\pi \in P_2^{\operatorname{bi}}(2p)}|R(\pi)|.
    \end{align}
    Since $R(\pi)$ is the set of all left-right noncrossing representations of $\pi$ and we sum over all $\pi \in P_2^{\operatorname{bi}}(2p)$, we get
    \begin{align*}
        \sum_{\pi \in P_2^{\operatorname{bi}}(2p)}|R(\pi)|=|\{(\pi_1,\pi_2,I,I^c): I \subseteq[2p], \pi_1 \in NC_2(I),\pi_2 \in NC_2(I^c)\}|.
    \end{align*}
    Note that for fixed $I \subseteq [2p]$, each $(\pi_1,\pi_2) \in NC_2(I) \times NC_2(I^c)$ will correspond to a unique partition $\pi \in P_2^{\operatorname{bi}}(2p)$ such that $(\pi_1,\pi_2,I,I^c)$ is a left-right noncrossing representation of $\pi$. Hence we can first sum over $I\subseteq [2p]$ so that 
    \begin{align*}
        \sum_{\pi \in P_2^{\operatorname{bi}}(2p)}|R(\pi)|=\sum_{I \subseteq [2p]}|\{(\pi_1,\pi_2): \pi_1 \in NC_2(I), \pi_2 \in NC_2(I^c)\}|.
    \end{align*}
    Therefore, by \eqref{eq: sum over bipartite partitions and representations}, we have
    \begin{align*}
        \sum_{\pi \in P_2^{\operatorname{bi}}(2p)}2^{\ccblocks(\pi)}=\sum_{I \subseteq [2p]}|\{(\pi_1,\pi_2): \pi_1 \in NC_2(I), \pi_2 \in NC_2(I^c)\}|.
    \end{align*}
    The latter cardinality can be written as a product over the cardinals and only depends on the length of $I$. More precisely, we have
    \begin{align*}
        \sum_{\pi \in P_2^{\operatorname{bi}}(2p)}2^{\ccblocks(\pi)}=\sum_{l=0}^p \binom{2p}{2l}C_lC_{p-l},
    \end{align*}
    where the cardinal $|NC_2(2l)|$ is equal to $C_l$. Hence
    \begin{align*}
        \int x^{2p}\, \text{d}\mu_1=m_p,
    \end{align*}
    and the first statement follows. For the second, we use the bijection $\Phi$ in \eqref{equation: mapping pi connected components} to write $\Phi(\pi)=\left(\hat{\pi},(\pi_T)_{T \in \hat{\pi}}\right)$. We note that
    \begin{align*}
        &\sum_{\substack{T \in \hat{\pi}\\ |T| \ge 4}}1=\ccblocks(\pi)-\ncrblocks(\pi),\\
        &\sum_{\substack{T \in \hat{\pi}\\ |T| \ge 4}}\left(\frac{|T|}{2}\right)=\crblocks(\pi).
    \end{align*}
    Hence
    \begin{align*}
        \sum_{\pi \in P_2^{\operatorname{bi}}(2p)}2^{\ccblocks(\pi)-p}&=\sum_{\pi \in P_2^{\operatorname{bi}}(2p)}2^{\ccblocks(\pi)-\ncrblocks(\pi)}\left(\frac{1}{2}\right)^{\crblocks(\pi)}\\
        &=\sum_{\hat{\pi} \in NC(2p)}\sum_{\substack{T \in \hat{\pi}\\ \pi_T \in P_2^{\operatorname{bicon}}(T)}}\prod_{\substack{T \in \hat{\pi}\\ |T| \ge 4}}2 \prod_{\substack{T \in \hat{\pi}\\ |T| \ge 4}}\left(\frac{1}{2}\right)^{|T|/2},
    \end{align*}
    where we use that $\ncrblocks(\pi)+\crblocks(\pi)=p$ in the first equality. We thus have
    \begin{align*}
        \sum_{\pi \in P_2^{\operatorname{bi}}(2p)}2^{\ccblocks(\pi)-p}=\sum_{\hat{\pi} \in NC(2p)}\prod_{\substack{T \in \hat{\pi}\\ |T| \ge 4}}\left\{2\left(\frac{1}{2}\right)^{|T|/2}|P_2^{\operatorname{bicon}}(|T|)|\right\}.
    \end{align*}
    Since the free cumulants $\cumuf_n(\mu_1)$ are uniquely characterized by the moments, it follows that for any odd integer $n$, we have $\cumuf_n(\mu_1)=0$, $\cumuf_2(\mu_1)=1$ and for any even integer $n \ge 4$, we have
    \begin{align*}
        \cumuf_n(\mu_1)=2\left(\frac{1}{2}\right)^{n/2}|P_2^{\operatorname{bicon}}(n)|.
    \end{align*}
\end{proof}
\begin{remark}
    For any $p \ge 1$, it was shown in \cite{gouyoubeauchamps1986} that
    \begin{align*}
        \sum_{l=0}^p \binom{2p}{2l}C_lC_{p-l}=C_pC_{p+1}.
    \end{align*}
    Therefore, the $2p$-th moment of $\mu_1$ is also characterized by $2^{-p}C_pC_{p+1}$. The sequence $(C_pC_{p+1})_{p \ge 1}$ is A005568 in Sloane's encyclopedia (\url{https://oeis.org/A005568}), where several combinatorial objects counted by it are shown.
\end{remark}

We are ready to compute the moments and free cumulants of $\mu_q$.
\begin{proposition}\label{prop: moments of mu_q}
Given $q\in [0,1]$, let   
    \begin{equation*}
        \mu_q:=\sqrt{q}\left(\frac{1}{\sqrt{2}}\mu_{sc}+\frac{1}{\sqrt{2}}\mu_{sc}\right)\boxplus\sqrt{1-q}\, \mu_{sc}.
    \end{equation*}
Then the following hold.
    \begin{enumerate}
        \item The odd free cumulants of $\mu_q$ vanish, $\cumuf_2(\mu_q)=1$ and for any even integer $n \ge 4$, we have
        \begin{align*}
            \cumuf_n(\mu_q)=2\Big(\frac{q}{2}\Big)^{n/2}|P_2^{\operatorname{bicon}}(n)|.
        \end{align*}
         \item The odd moments of $\mu_q$ vanish and, for every $p\in \N$, its $2p$-th moment is given by  
        \begin{align*}
       \sum_{\pi \in P_2^{\operatorname{bi}}(2p)}2^{\ccblocks(\pi)-p}q^{\crblocks(\pi)}.
        \end{align*}
    \end{enumerate}
\end{proposition}

\begin{proof}
    We will use the notion of $R$-transform; see \cite[Lecture 16]{nica2006lectures}. For a random variable $a \in \A$ with distribution $\mu$, let
\begin{align*}
    R_a(z)=R_\mu(z)=\sum_{n \ge 1}\cumuf_n(a) z^n,
\end{align*}
be its $R$-transform, defined as a formal series. 
The $R$-transform of a standard semi-circle law is given by
\begin{align*}
    R_{\mu_{sc}}(z)=z^2,
\end{align*}
whereas Proposition \ref{proposition: cumulants and moments mu_1} shows that the $R$-transform of $\mu_1$ is given by
\begin{align*}
    R_{\mu_1}(z)=z^2+2\sum_{n \ge 2}\left(\frac{1}{2}\right)^{n}|P_2^{\operatorname{bicon}}(2n)|z^{2n}.
\end{align*}
Since the free cumulants linearize the free convolution, we deduce that 
\begin{align*}
R_{\mu_q}(z)&= R_{\sqrt{q}\mu_1}(z)+ R_{\sqrt{1-q} \,\mu_{sc}}(z) \\
& = R_{\mu_1}\big(\sqrt{q}\,z\big) + R_{\mu_{sc}}\big(\sqrt{1-q} \,z\big)\\
&= z^2+ 2\sum_{n\ge 2} \Big(\frac{q}{2}\Big)^{n}|P_2^{\operatorname{bicon}}(2n)|z^{2n}.
\end{align*}
This proves the first part of the proposition. 
To prove the second part, we use the moment-cumulant formula to deduce that the odd moments vanish, while the $2p$-th moment can be expressed as
\begin{align*}
\sum_{\hat{\pi}\in NC(2p)}\prod_{T \in \hat{\pi}} \cumuf_{|T|}(\mu_q)
&=\sum_{\hat{\pi}\in NC(2p)}\prod_{\underset{|T|\geq 4}{T \in \hat{\pi}}}\left\{ 2\Big(\frac{q}{2}\Big)^{|T|/2}|P_2^{\operatorname{bicon}}(T)|\right\}\\
&=\sum_{\hat{\pi}\in NC(2p)}\sum_{(\pi_T)_{T \in \hat{\pi}}}\prod_{\underset{|T|\geq 4}{T \in \hat{\pi}}} \left\{2\Big(\frac{q}{2}\Big)^{|T|/2}\right\},\\
\end{align*}
where the second summation is over $(\pi_T)_{T \in \hat{\pi}} \in \proj(\hat{\pi})$ such that $\pi_T \in P_2^{\operatorname{bicon}}(T)$ for $T \in \hat{\pi}$.
Now noting that 
\begin{align*}
    &|\{T \in \hat{\pi}:|T| \ge 4\}|=\ccblocks(\pi)-\ncrblocks(\pi), 
\end{align*}
and
$$
\sum_{\substack{T \in \hat{\pi}\\ |T| \ge 4}}\frac{|T|}{2}=\crblocks(\pi),
$$
we finish the proof after using the bijection $\Phi$ from \eqref{equation: mapping pi connected components} to rewrite the above expression.
\end{proof}



\section{Existence of the limit}\label{sec: existence}

The goal of this section is to show that the expression in \eqref{eq: intro-tensor-setting} admits a limit that depends only on the first and second moments of the variables at hand. 
Let us first prove the following centering lemma.
\begin{lemma}\label{lemma: centering}
    Let $(\A,\tau)$ be a unital faithful tracial noncommutative probability space. Let $\{a_k: k \in [n]\},\{d_k: k \in [n]\} \subset \A$ be two collections of free random variables and let
    \begin{align*}
        c_k:=a_k\otimes d_k-\tau\otimes \tau(a_k\otimes d_k),
    \end{align*}
    for every $k\in[n]$. Then for any $m \ge 1$ and $i_1,\ldots,i_m \in [n]$ such that there exists an index $l^* \in [m]$ satisfying $i_{l^*} \ne i_j$ for all $j \ne l^*$, we have
    \begin{align*}
        \tau\otimes \tau(c_{i_1}\cdots c_{i_{m}})=0.
    \end{align*}
    In particular, we have
    \begin{align*}
        \tau\otimes \tau\left(c_{i_1}\cdots c_{i_{l^*-1}}a_{i_{l^*}}\otimes d_{i_{l^*}} c_{i_{l^*+1}}\cdots c_{i_m} \right)=\tau(a_{i_{l^*}})\tau(d_{i_{l^*}})\tau\otimes \tau(c_{i_1}\cdots c_{i_{l^*-1}} c_{i_{l^*+1}}\cdots c_{i_m}).
    \end{align*}
\end{lemma}
\begin{proof}
    By cyclicity of the trace, we can assume that $l^*=1$. We write
    \begin{align}\label{eq: trace of c_ij}
        \tau\otimes \tau(c_{i_1}\cdots c_{i_m})=\tau\otimes \tau(a_{i_1}\otimes d_{i_1}c_{i_2}\cdots c_{i_m})-\tau\otimes \tau(a_{i_1}\otimes d_{i_1})\tau\otimes \tau(c_{i_2}\cdots c_{i_m}).
    \end{align}
    For the first term, we open the expressions for $c_{i_j}$ for $j \ge 2$ to get that
    \begin{align*}
        &\tau\otimes \tau(a_{i_1}\otimes d_{i_1}c_{i_2}\cdots c_{i_m}) \\
        &=\sum_{I\subseteq \{2,\ldots,m\}}(-1)^{|I|}\prod_{j \in I}\left(\tau\otimes \tau(a_{i_j}\otimes d_{i_j}) \right)\tau\otimes \tau\left(a_{i_1}\otimes d_{i_1}\prod_{j \in J^c}^{\to}a_{i_j}\otimes d_{i_j}\right),
    \end{align*}
    where $\overset{\to}{\prod}$ denotes the product respecting the ordering. Since $a_{i_1},d_{i_1}$ are free from $a_{i_j},d_{i_j}$ for all $j \ge 2$, freeness implies that
    \begin{align*}
         &\tau\otimes \tau(a_{i_1}\otimes d_{i_1}c_{i_2}\cdots c_{i_m})\\
         &=\tau\otimes \tau(a_{i_1}\otimes d_{i_1})\sum_{I\subseteq \{2,\ldots,m\}}(-1)^{|I|}\prod_{j \in I}\left(\tau\otimes \tau(a_{i_j}\otimes d_{i_j})\right) \tau\otimes \tau\left(\prod_{j \in J^c}^{\to}a_{i_j}\otimes d_{i_j}\right).
    \end{align*}
    The summation can be easily written as $\tau\otimes \tau(c_{i_2}\cdots c_{i_m})$, hence
    \begin{align*}
        \tau\otimes \tau(a_{i_1}\otimes d_{i_1}c_{i_2}\cdots c_{i_m})=\tau\otimes \tau(a_{i_1}\otimes d_{i_1})\tau\otimes \tau(c_{i_2}\cdots c_{i_m}).
    \end{align*}
    By \eqref{eq: trace of c_ij}, we then have
    \begin{align*}
        \tau\otimes \tau(c_{i_1}\cdots c_{i_m})=0.
    \end{align*}
    Finally, since $\tau\otimes \tau(a_{i_1}\otimes d_{i_1})=\tau(a_{i_1})\tau(d_{i_1})$, we get that
    \begin{align*}
        \tau\otimes \tau(a_{i_1}\otimes d_{i_1}c_{i_2}\cdots c_{i_m})&=\tau\otimes \tau(a_{i_1}\otimes d_{i_1})\tau\otimes \tau(c_{i_2}\cdots c_{i_m})\\
        &=\tau(a_{i_1})\tau(d_{i_1})\tau\otimes \tau(c_{i_2}\cdots c_{i_m}).
    \end{align*}
\end{proof}

We are now ready to prove the existence of the limit and that it only depends on the first and second moments of the $a_k$'s.

\begin{proposition}[Existence]\label{proposition: existence of the limit}
    Let $(a_n)_{n\in \N} \in (\A,\tau)$ be self-adjoint free i.i.d random variables with mean $\lambda$,  variance $\sigma^2$, and denote $\delta^2:=\var(a_1\otimes a_1)= \sigma^2(\sigma^2+2\lambda^2)$. For every $k\in \N$, denote 
    \begin{align*}
        b_k:=\frac{1}{\delta}\big(a_k\otimes a_k-\tau\otimes \tau(a_k\otimes a_k)\big),
    \end{align*}
    and for every $n\in\N$, denote
    \begin{align*}
        S_n:=\frac{1}{\sqrt{n}}\sum_{k \in [n]}b_k.
    \end{align*}
    Then there exists a random variable ${\bf{S}} \in (\A',\tau')$ such that $S_n\Rightarrow {\bf{S}}$. Moreover, the law of ${\bf{S}}$ depends only on $\lambda$ and $\sigma$, its odd moments vanish and, for every $p\in\N$, 
    \begin{align*}
        \tau'({\bf{S}}^{2p})=\sum_{\pi \in P_2(2p)}\tau\otimes \tau(b_{i_1}\cdots b_{i_{2p}}),
    \end{align*}
    where  $i \in [p]^{2p}$ is any sequence such that $i_j=i_k$ if and only $\{j,k\} \in \pi$.
\end{proposition}

\begin{proof}
  We begin by writing
    \begin{align*}
        \tau\otimes \tau(S_n^p)=\frac{1}{n^{p/2}}\sum_{i \in [n]^p}\tau\otimes \tau (b_{i_1}\cdots b_{i_p}).
    \end{align*}
    Since $b_{1},\ldots,b_n$ are identically distributed, the expression
    \begin{align}\label{equation: def of tau(pi)}
        \tau\otimes \tau (b_{i_1}\cdots b_{i_p})
    \end{align}
    depends only on the partition $\pi=\pi(i) \in P(p)$ given by $l \sim_\pi k $ (that is, $l,k$ belong to the same block of $\pi$) if and only if $i_l=i_k$. Denote the common value of \eqref{equation: def of tau(pi)} by $\tau\otimes \tau(\pi)$. Then we have
    \begin{align*}
        \tau\otimes \tau(S_n^p)=\frac{1}{n^{p/2}}\sum_{\pi \in P(p)}\tau\otimes \tau(\pi)|\{i \in [n]^p: \pi(i)=\pi\}|.
    \end{align*}
    To count the cardinality, we choose an index for each block. Therefore, we have
    \begin{align*}
        |\{i \in [n]^p: \pi(i)=\pi\}|=n(n-1)\cdots (n-|\pi|+1) \sim n^{|\pi|}.
    \end{align*}
    By Lemma \ref{lemma: centering}, if $\pi$ has a block of size $1$, $\tau\otimes \tau(\pi)=0$. Thus, we have
    \begin{align*}
        \tau\otimes \tau(S_n^p)=\sum_{\substack{\pi \in P(p)\\ |V| \ge 2; \forall V \in \pi}}\tau\otimes \tau(\pi)\frac{n(n-1)\cdots (n-|\pi|+1) }{n^{p/2}}.
    \end{align*}
    Since $|V| \ge 2$ for all blocks $V \in \pi$, we have $|\pi| \le p/2$. If there exists a block $V \in \pi$ such that $|V| \ge 3$, we immediately have $|\pi|<p/2$, and its contribution is negligible. In particular, this implies that the odd moments of $S_n$ are asymptotically vanishing. We deduce that
    \begin{align*}
        \lim_{n \to \infty}\tau\otimes \tau(S_n^p)=\lim_{n\to \infty}\sum_{\substack{\pi \in P(p)\\ |V| = 2; \forall V \in \pi}}\tau\otimes \tau(\pi)\frac{n(n-1)\cdots (n-|\pi|+1) }{n^{p/2}}.
    \end{align*}
    In this case, $\pi$ is a pair partition and $|\pi|=p/2$, hence we deduce the formula
    \begin{align*}
        \lim_{n \to \infty}\tau\otimes \tau(S_n^p)=\sum_{\pi \in P_2(p)}\tau\otimes \tau(\pi),
    \end{align*}
which shows that $S_n$ converges. 
To prove that the limit depends only on $\lambda$ and $\sigma$, we write
\begin{align*}
        \tau\otimes \tau(\pi)=\tau\otimes \tau(b_{i_1}\cdots b_{i_p})=\frac{1}{\delta^p}\sum_{I \subseteq [p]}(-1)^{|I|}\lambda^{2|I|}\tau^2\left(\prod_{l \in I^c}^{\to}a_{i_l}\right),
    \end{align*}
    where $\pi(i)=\pi$. By the moment-cumulant formula, we have
    \begin{align*}
        \tau\left(\prod_{l \in I^c}^{\to}a_{i_l}\right)=\sum_{\sigma \in NC(I^c)} \kappa_\sigma((a_{i_l})_{l \in I^c}).
    \end{align*}
    Since the mixed cumulants of free variables vanish, the only partitions $\sigma \in NC(I^c)$ that contribute are those such that every block $V \in \sigma$ has cardinality at most two. We then have
    \begin{align*}
        \kappa_\sigma((a_{i_l})_{l \in I^c})&=\prod_{\substack{V \in \sigma\\ |V|=2}}\kappa_{|V|}(a,a)\prod_{\substack{V \in \sigma\\ |V|=1}}\kappa_{|V|}(a) =\sigma^{2\#\{V \in \sigma:|V|=2\}}\lambda^{\#\{V \in \sigma:|V|=1\}}.
    \end{align*}
    This concludes the proof.
\end{proof}


\section{Proof of Theorem~\ref{th: main theorem}}\label{section: proof}

After proving the existence of the limit in the previous section, the goal here is to identify this limit as stated in Theorem~\ref{th: main theorem}. In all this section, $(a_n)_{n\in \N}$ denote free copies of a random variable $a$ with mean $\lambda$ and variance $\sigma^2$. Moreover, the common law of the normalized tensors will be denoted by 
\begin{align}\label{equation: common distribution of b}
    b=\frac{1}{\delta}\big(a\otimes a-\lambda^2\big),
\end{align}
where $\delta^2=\var(a\otimes a)=\sigma^2(\sigma^2+2\lambda^2)$. 

Throughout the proof, we will assume $p$ is an even integer. Following Proposition~\ref{proposition: existence of the limit}, we denote ${\bf{S}}$ the limit of $S_n$ and note that 
\begin{align*}
    \tau'({\bf{S}}^{p})=\sum_{\pi \in P_2(p)}\tau\otimes \tau(\pi),
\end{align*}
where $\tau\otimes \tau(\pi)=\tau\otimes \tau(b_{i_1}\cdots b_{i_p})$ with $\pi(i)=\pi$. 

\subsection{Contribution of noncrossing blocks}

We begin by removing interval blocks.
\begin{lemma}\label{lemma: interval blocks}
    Let $\pi \in P_2(p)$ and suppose that there exists $l \in [p]$ such that $\{l,l+1\} \in \pi$ (with the convention that $p+1:=1$). Then
    \begin{align*}
        \tau \otimes \tau (\pi)=\tau\otimes \tau(\pi\setminus \{l,l+1\}).
    \end{align*}
\end{lemma}
\begin{proof}
    By cyclicity, we can assume $l=p-1$. We then have
    \begin{align*}
       \delta^2 \tau\otimes \tau(\pi)&=\tau\otimes \tau\left(b_{i_1}\cdots b_{i_{p-2}}\cdot \left(a_{i_p}^2\otimes a_{i_p}^2+\lambda^41\otimes 1-2\lambda^2a_{i_p}\otimes a_{i_p}\right)\right)\\
        &=\tau\otimes \tau\left(b_{i_1}\cdots b_{i_{p-2}}\cdot a_{i_p}^2\otimes a_{i_p}^2\right)+\lambda^4\tau\otimes \tau\left(b_{i_1}\cdots b_{i_{p-2}}\right)\\
        &\quad -2\lambda^2\tau\otimes \tau\left(b_{i_1}\cdots b_{i_{p-2}}a_{i_p}\otimes a_{i_p}\right)\\
        &=:I+II+III.
    \end{align*}
    We immediately recognize
    \begin{align*}
        II=\lambda^4\tau\otimes \tau(\pi\setminus \{p-1,p\}).
    \end{align*}
    By Lemma \ref{lemma: centering}, we have
    \begin{align*}
        &I=\tau^2(a^2)\tau\otimes \tau(\pi\setminus \{p-1,p\});\\
        &III=-2\lambda^4\tau\otimes \tau(\pi\setminus \{p-1,p\}).
    \end{align*}
    Hence
    \begin{align*}
        \delta^2\tau\otimes \tau(\pi)=\left(\tau^2(a^2)-\lambda^4\right)\tau\otimes \tau(\pi\setminus \{p-1,p\}).
    \end{align*}
    To conclude, we note that $\tau^2(a^2)-\lambda^4=\delta^2$ and finish the proof. 
\end{proof}

Recall that a noncrossing pair partition $\pi \in NC_2(p)$ always has an interval block $V=\{l,l+1\} \in \pi$ such that $\pi\setminus V$ is a noncrossing pair partition. In particular, by induction, Lemma \ref{lemma: interval blocks} implies the following.
\begin{corollary}\label{corollary: contribution of noncrossing partitions}
    For any $\pi \in NC_2(p)$, we have $\tau\otimes \tau(\pi)=1$.
\end{corollary}

\subsection{Decomposition of pair partitions}\label{subsection: graph of crossings}
In order to capture the contribution of crossing partitions, we need to decompose a partition $\pi \in P_2(p)\setminus NC_2(p)$ using smaller partitions. We denote $\pi=\pi_1\oplus \cdots \oplus \pi_k$ if $[p]$ can be decomposed into $k$ intervals $I_1,\ldots,I_k$ such that $\pi_k \in P_2(I_k)$ and $V \in \pi$ if $V \in \pi_l$ for some $1 \le l \le k$. 
For $I \subseteq [p]$, let
\begin{align}\label{equation: notation a_I}
    a_I:=\prod_{l \in I}^{\to}a_{i_l},
\end{align}
and $a_{\varnothing}:=1$.

\begin{lemma}\label{lemma: recursion tau(pi)}
Let $\pi \in P_2(p)$. Then, the following hold.
\begin{enumerate}[label=(\ref{lemma: recursion tau(pi)}.\roman*)]
    \item\label{recursion: direct sum} If $\pi=\pi_1 \oplus \cdots \oplus \pi_l$, then
    \begin{align*}
        \tau\otimes \tau(\pi)=\tau\otimes \tau(\pi_1)\cdots \tau\otimes \tau(\pi_l).
    \end{align*}
    \item\label{recursion: superblocks} If $\pi=\{1,p\}\cup \pi_1$, where $\pi_1 \in P_2(\{2,\ldots,p-1\})$, then 
    \begin{align*}
        \tau\otimes \tau(\pi)=\tau\otimes \tau(\pi_1).
    \end{align*}
    Moreover, if there exists an interval $I \subseteq [p]$ such that $\pi|_I$ is a pair partition, then
    \begin{align*}
        \tau\otimes \tau(\pi)=\tau\otimes \tau(\pi|_I)\tau\otimes \tau (\pi|_{I^c}).
    \end{align*}
    \item\label{recursion: superblock with direct sum} If $\pi$ has a block $V=\{r,s\}$ such that for any block $U=\{l,k\} \in \pi$ with $r <l<s$, we have $r<k<s$ (i.e., every point inside $V$ matches another one inside $V$), we have
    \begin{align*}
        \tau\otimes \tau(\pi)=\tau\otimes \tau(\pi|_{V_-})\tau\otimes \tau(\pi|_{V_+}),
    \end{align*}
    where $\pi|_{V_-}$ is the restriction of $\pi$ to inside of $V$ and $\pi|_{V_+}$ is the restriction of $\pi$ to outside of $V$.
\end{enumerate}
\end{lemma}
\begin{proof}
    \ref{recursion: direct sum} By induction, it suffices to prove the case $\pi=\pi_1\oplus \pi_2$. Let $I_1,I_2$ be the disjoint decomposition of $[p]$ given by $\pi_1$ and $\pi_2$. Then, we can write
    \begin{align*}
        \tau\otimes \tau(\pi)=\frac{1}{\delta^p}\sum_{\substack{J_1 \subseteq I_1\\ J_2 \subseteq I_2}}(-1)^{|J_1|+|J_2|}\lambda^{2(|J_1|+|J_2|)}\tau^2\left(\prod_{l \in J_1^c}a_{i_l}\prod_{l \in J_2^c}a_{i_l}\right).
    \end{align*}
    Since $\pi$ is the direct sum of $\pi_1,\pi_2$, the variables $a_{J_1^c}, a_{J_2^c}$ are free and we can write
    \begin{align*}
        \tau\otimes \tau(\pi)=\frac{1}{\delta^p}\sum_{\substack{J_1 \subseteq I_1\\ J_2 \subseteq I_2}}(-1)^{|J_1|+|J_2|}\lambda^{2(|J_1|+|J_2|)}\tau^2\left(\prod_{l \in J_1^c}a_{i_l}\right)\tau^2\left(\prod_{l \in J_2^c}a_{i_l}\right).
    \end{align*}
    It is immediate to check that the right-hand-side is equal to $\tau\otimes \tau(\pi_1)\tau\otimes \tau(\pi_2)$.

    \ref{recursion: superblocks} By cyclicity, we have
    \begin{align*}
        \tau\otimes \tau(\pi)&=\tau\otimes \tau(b_{i_2}\cdots b_{i_{p-1}}b_{i_{p}}b_{i_1}).
    \end{align*}
    Since $\{1,p\} \in \pi$, Lemma \ref{lemma: interval blocks} implies that
    \begin{align*}
        \tau\otimes \tau(\pi)=\tau\otimes \tau(b_{i_2}\cdots b_{i_{p-1}})=\tau\otimes \tau(\pi_1).
    \end{align*}
    The second part follows again by cyclicity as we can assume $I=\{1,\ldots,k\}$ for some $k \in [p]$, and the variables are free.

    \ref{recursion: superblock with direct sum} By cyclicity, we can assume that $V=\{1,k\}$ for some $k\in [p]$. Note that $V$ creates a direct sum $\pi=(V\cup \pi|_{V_-})\oplus \pi|_{V_+}$. The result follows by \ref{recursion: direct sum} and \ref{recursion: superblocks}.
\end{proof}

Note that any block $V \in \pi$ that does not cross any other block of $\pi$ is either an interval block or a block that satisfies \ref{recursion: superblock with direct sum}. In particular, its removal does not affect the value of $\tau\otimes \tau(\pi)$. It is clear then that $\tau\otimes \tau(\pi)$ is a multiplicative function \cite{BS1996} over the connected components of $\pi$. Using the mapping $\Phi$ defined in \eqref{equation: mapping pi connected components}, we can write
\begin{align}\label{equation: decomposition of tau in connected parts}
    \tau\otimes \tau(\pi)=\prod_{T \in \hat{\pi}}\tau\otimes \tau (\pi_T),
\end{align}
where $(\hat{\pi},(\pi_T)_{T \in \hat{\pi}})=\Phi(\pi)$. 
In view of this, we will now focus on the case where $\pi \in P_2^{\operatorname{con}}(p)$, for $p \ge 4$, as the case $p=2$ corresponds to noncrossing blocks.

\subsection{Contribution of connected partitions}\label{subsection: contribution of connected partitions}

Given $\pi\in P_2^{\operatorname{con}}(p)$, for an even integer $p \ge 4$, we recall that $\pi \in P_2^{\operatorname{bicon}}(p)$ if its intersection graph $G(\pi)$ is a bipartite connected graph.

The following is the main proposition of this subsection.
\begin{proposition}\label{proposition: contribution bipartite connected part}
    Let $p \ge 4$ be an even integer and $\pi \in P_2^{\operatorname{con}}(p)$. Then the following hold.
    \begin{enumerate}
        \item If $\pi \notin P_2^{\operatorname{bicon}}(p)$, then $\tau\otimes \tau(\pi)=0$ ;
        \item If $\pi \in P_2^{\operatorname{bicon}}(p)$, then
        \begin{align*}
        \tau\otimes \tau(\pi)=2 \Big(\frac{q}{2}\Big)^{\frac{p}{2}},
        \end{align*}
        where $q=\frac{2\lambda^2}{\sigma^2+2\lambda^2}$.
    \end{enumerate}    
\end{proposition}

To prove Proposition \ref{proposition: contribution bipartite connected part}, our goal is to remove the blocks $V \in \pi$ one at a time and track its influence on the rest of the partition $\pi\setminus V$. To this end, we will define two main quantities associated with $V$. First, we will define a coloring $w_V \in \{0,1\}$ of $V$ that encodes how the removal of $V$ affects the other blocks. Secondly, we will define a binary vector $\theta_V \in \{0,1\}^4$ satisfying
\begin{align*}
    \sum_{j\in[4]} \theta_j=1,
\end{align*}
that is, only one coordinate of $\theta_V$ is equal to $1$. It encodes how we remove the block $V$ from $\pi$, that is, its weight to $\tau\otimes\tau(\pi)$. Let us begin this description now.

Let $\pi \in P_2^{\operatorname{con}}(p)$ and $i \in [p]^p$ such that $\pi(i)=\pi$. Given an integer $t \ge 0$,  blocks $V_1,\ldots,V_t,V_{t+1} \in \pi$, and a coloring $w \in \{0,1\}^{t+1}$, we denote $V_j^{(w_j)}$ the block $V_j$ under the color $w_j$. We then define the iterated joint law of $b_{i_j}$
\begin{align*}
    B_k:=B_k\left(V_1^{(w_1)},\ldots, V_k^{(w_k)}\right)=\left(b_{i_j,k}\right)_{j \in [p]}
\end{align*}
for $k \le t+1$ as follows. Let $(\tilde{a}_j)_{j \in \N}$ be a family of free independent copies of $a$, free from $(a_j)_{j \in \N}$. For $k=0$, we define $B_0=(b_{i_j})_{j \in [p]}$ and for any $1 \le k \le t+1$, we have
\begin{align*}
    B_k=\frac{1}{\delta}\left(a_{i_j}^{(l)}\otimes a_{i_j}^{(r)}-\lambda^2\right)_{j \in [p]},
\end{align*}
where
\begin{align}\label{equation: a_j is either a or tilde a}
    a_{i_j}^{(l)},a_{i_j}^{(r)} \in \{a_{i_j},\tilde{a}_{i_j}\},
\end{align}
for any $j \in [p]$. We then recursively define the choices in \eqref{equation: a_j is either a or tilde a} as follows. Given 
\begin{align*}
    B_k=\frac{1}{\delta}\left(a_{i_j}^{(l)}\otimes a_{i_j}^{(r)}-\lambda^2\right)_{j \in [p]},
\end{align*}
we consider the block $V_{k+1}=\{s_1<s_2\}$ and its color $w_{k+1}$, for $1 \le k \le t$. If $w_{k+1}=0$, for all blocks $V=\{t_1,t_2\} \in \pi$ that cross $V_{k+1}$, we replace the left legs of $(b_{i_{t_1},k},b_{i_{t_2},k})$ by a pair of free variables. Concretely, we define
    \begin{align*}
        &b_{i_{t_1},k+1}=\frac{1}{\delta}\left(\tilde{a}_{i_{t_1}}\otimes a_{i_{t_1}}^{(r)}-\lambda^2\right);\\
        &b_{i_{t_2},k+1}=\frac{1}{\delta}\left(a_{i_{t_1}}\otimes a_{i_{t_2}}^{(r)}-\lambda^2\right).
    \end{align*}
Otherwise, if $w_{k+1}=1$, we replace the right legs of $(b_{i_{t_1},k},b_{i_{t_2},k})$ by a pair of free variables instead,
\begin{align*}
    &b_{i_{t_1},k+1}=\frac{1}{\delta}\left(a_{i_{t_1}}^{(l)}\otimes \tilde{a}_{i_{t_1}}-\lambda^2\right);\\
    &b_{i_{t_2},k+1}=\frac{1}{\delta}\left(a_{i_{t_2}}^{(l)}\otimes a_{i_{t_1}}-\lambda^2\right).
\end{align*}
All the other blocks remain unchanged. Recall that for a vector $B$ and a set $I \subseteq [p]$, we set as in \eqref{equation: notation a_I}
\begin{align*}
    B_{I}=\prod_{l \in I}^{\to} B_{i_l},
\end{align*}
and $B_{\varnothing}=1$.
We then define
\begin{align*}
    &\tau\otimes\tau\left(\pi,V_1^{(w_1)},\ldots,V_k^{(w_k)}\right):=\tau\otimes \tau\left(\left(B_{k}\right)_{[p]}\right)\\
    &\tau\otimes\tau \left(\pi\setminus \{U_1,\ldots,U_{m}\},V_1^{(w_1)},\ldots,V_{k}^{(w_{k})}\right):=\tau\otimes \tau\left(\left(B_k\right)_{[p]\setminus \left(U_1\cup\cdots\cup U_{m}\right)}\right),
\end{align*}
where $U_1,\ldots,U_{m},V_1,\ldots,V_{k} \in \pi$ and
\begin{align*}
    B_{k}=B_k\left(V_1^{(w_1)},\dots,V_{k}^{(w_k)}\right).
\end{align*}

Note first that at each step $k$, the distribution of $(b_{i_{v_1},k},b_{i_{v_2},k})$ for a block $V=\{v_1<v_2\} \in \pi$ can be written as one of the following four cases.
\begin{enumerate}
    \item\label{case no crossing} The blocks $V_1,\ldots,V_k$ do not cross $V$, and we have
    \begin{align*}
        &b_{i_{v_1},k}=\frac{1}{\delta}\left(a_{i_{v_1}}\otimes a_{i_{v_1}}-\lambda^2\right);\\
        &b_{i_{v_2},k}=\frac{1}{\delta}\left(a_{i_{v_1}}\otimes a_{i_{v_1}}-\lambda^2\right),
    \end{align*}
    where we recall that $i_{v_1}=i_{v_2}$.
    \item\label{case cross w=0} All blocks $V_j$ that cross $V$ have the same color $w_j=0$, and we have
    \begin{align*}
        &b_{i_{v_1},k}=\frac{1}{\delta}\left(\tilde{a}_{i_{v_1}}\otimes a_{i_{v_1}}-\lambda^2\right);\\
        &b_{i_{v_2},k}=\frac{1}{\delta}\left(a_{i_{v_1}}\otimes a_{i_{v_1}}-\lambda^2\right).
    \end{align*}
    \item\label{case cross w=1} All blocks $V_j$ that cross $V$ have the same color $w_j=1$, and we have
    \begin{align*}
        &b_{i_{v_1},k}=\frac{1}{\delta}\left(a_{i_{v_1}}\otimes \tilde{a}_{i_{v_1}}-\lambda^2\right);\\
        &b_{i_{v_2},k}=\frac{1}{\delta}\left(a_{i_{v_1}}\otimes a_{i_{v_1}}-\lambda^2\right).
    \end{align*}
    \item\label{case cross both w} Otherwise, there exist two blocks $V_{j_1},V_{j_2}$ that cross $V$ such that $w_{j_1}\ne w_{j_2}$, that is, they have different colors, and then
    \begin{align*}
        &b_{i_{v_1},k}=\frac{1}{\delta}\left(\tilde{a}_{i_{v_1}}\otimes \tilde{a}_{i_{v_1}}-\lambda^2\right);\\
        &b_{i_{v_2},k}=\frac{1}{\delta}\left(a_{i_{v_1}}\otimes a_{i_{v_1}}-\lambda^2\right).
    \end{align*}
\end{enumerate}
We remark that even if the choices of replacements do not always happen for $b_{i_{v_1}}$, the joint distribution of $(b_{i_{v_1},k},b_{i_{v_2},k})$ can always be written as one of those four cases since the variables are free i.i.d. For instance, we have
\begin{align*}
    (a\otimes a,a\otimes \tilde{a})\overset{d}{=}(a\otimes \tilde{a},a\otimes a)\overset{d}{=}(\tilde{a}\otimes a,\tilde{a}\otimes \tilde{a}),
\end{align*}
and so on. Let
\begin{align*}
    &\theta \in \{0,1\}^4,\\
    &\sum_{j \in [4]}\theta_j=1,
\end{align*}
and denote 
\begin{align*}
    &b_{i_{v_1}}(\theta)=\frac{1}{\delta}\left(\theta_1 a_{i_{v_1}}\otimes a_{i_{v_1}}+\theta_2 \tilde{a}_{i_{v_1}}\otimes a_{i_{v_1}}+\theta_3 a_{i_{v_1}}\otimes \tilde{a}_{i_{v_1}}+\theta_4 \tilde{a}_{i_{v_1}}\otimes \tilde{a}_{i_{v_1}}-\lambda^2\right);\\
    &b_{i_{v_2}}(\theta)=\frac{1}{\delta}\left(a_{i_{v_1}}\otimes a_{i_{v_1}}-\lambda^2\right).
\end{align*}
The choice of $\theta:=\theta\left(V, V_1^{(w_1)}\ldots, V_k^{(w_k)}\right)$ corresponds to the case in which we are, that is, we define $\theta$ as the binary vector such that
\begin{align*}
    \left(b_{i_{v_1}}(\theta),b_{i_{v_2}}(\theta)\right)\overset{d}{=}\left(b_{i_{v_1},k},b_{i_{v_2},k}\right).
\end{align*}
Let $V=\{v_1<v_2\} \in \pi\setminus\{V_1,\ldots, V_k\}$ and
\begin{align*}
    &I_1:=\{1,\ldots,v_1-1\}\setminus \left(\bigcup_{j \in [k]}V_j\right);\\
    &I_2:=\{v_1+1,\ldots,v_2-1\}\setminus \left(\bigcup_{j \in [k]}V_j\right);\\
    &I_3:=\{v_2+1,\ldots,p\}\setminus \left(\bigcup_{j \in [k]}V_j\right).
\end{align*}
Our first goal is to obtain the contribution and coloring of $V$ to
\begin{align*}
    \tau\otimes\tau \left(\pi\setminus \{V_1,\ldots,V_{k}\},V_1^{(w_1)},\ldots,V_{k}^{(w_{k})}\right)=\tau\otimes \tau\left(\left(B_k\right)_{I_1}b_{i_{v_1},k}\left(B_k\right)_{I_2}b_{i_{v_2},k}\left(B_k\right)_{I_3}\right),
\end{align*}
according to the value of $\theta= \theta\left(V, V_1^{(w_1)}\ldots, V_k^{(w_k)}\right)\in \{0,1\}^4$, where 
\begin{align*}
    B_k=B_k\left(V_1^{(w_1)},\ldots,V_k^{(w_k)}\right).
\end{align*}
We recall that
\begin{align*}
    \tau\otimes\tau\left(\left(B_k\right)_{I_1}\left(B_k\right)_{I_2}\left(B_k\right)_{I_3}\right)&=\tau\otimes \tau\left(\left(B_k\right)_{[p]\setminus \left(V_1\cup \cdots \cup V_k \cup V\right)}\right)\\
    &=\tau\otimes \tau\left(\pi\setminus\{V_1,\ldots,V_k,V\},V_1^{(w_1)},\ldots,V_k^{(w_k)}\right)
\end{align*}
To simplify the notation, we denote
\begin{align}\label{eq: power of delta}
    h:=|I_1|+|I_2|+|I_3|=p-2(k+1).
\end{align}

We divide the four cases into the next four lemmas.
\begin{lemma}[Case $\theta_1=1$]\label{lemma: case theta_1=1}
    Let $p \ge 4$ be an even integer, $\pi \in P_2^{\operatorname{con}}(p)$. Let $V_1,\ldots,V_k \in \pi$ be blocks, $w \in \{0,1\}^k$ be a coloring of $V_1,\ldots,V_k$ and
    \begin{align*}
        B_k=B_k\left(V_1^{(w_1)},\ldots,V_k^{(w_k)}\right)=\left(b_{i_j,k}\right)_{j \in [p]}.
    \end{align*}
    Let $V=\{v_1<v_2\} \in \pi\setminus \{V_1,\ldots,V_{k}\}$ and the joint law 
    \begin{align*}
        \left(b_{i_{v_1},k},b_{i_{v_2},k}\right)\overset{d}{=}\left(b_{i_{v_1}}(\theta),b_{i_{v_2}}(\theta)\right).
    \end{align*}
    If $\theta_1=1$, we have
    \begin{align*}
        &\tau\otimes\tau \left(\pi\setminus \{V_1,\ldots,V_{k}\},V_1^{(w_1)},\ldots,V_{k}^{(w_{k})}\right)\\
        &=\frac{q}{2}\sum_{w=0,1}\tau\otimes\tau \left(\pi\setminus \{V_1,\ldots,V_k,V\},V_1^{(w_1)},\ldots,V_{k}^{(w_{k})},V^{(w)}\right).
    \end{align*}
\end{lemma}
\begin{proof}
    To simplify the notation, let
    \begin{align*}
        T:=\tau\otimes\tau \left(\pi\setminus \{V_1,\ldots,V_{k}\},V_1^{(w_1)},\ldots,V_{k}^{(w_{k})}\right).
    \end{align*}
    Then, we must prove that
    \begin{align*}
        T=\frac{q}{2}\sum_{w=0,1}\tau\otimes\tau \left(\pi\setminus \{V_1,\ldots,V_k,V\},V_1^{(w_1)},\ldots,V_{k}^{(w_{k})},V^{(w)}\right).
    \end{align*}
    As $\theta_1=1$, we can write
    \begin{align*}
        &b_{i_{v_1},k}=\frac{1}{\delta}\left(a\otimes a-\lambda^2\right);\\
        &b_{i_{v_2},k}=\frac{1}{\delta}\left(a\otimes a-\lambda^2\right),
    \end{align*}
    where $a:=a_{i_{v_1}}$. We get by Lemma \ref{lemma: centering} that
    \begin{align}\label{eq: value of T}
        T&=\frac{1}{\delta^2}\tau\otimes \tau\left\{\left(B_k\right)_{I_1}a\otimes a\left(B_k\right)_{I_2}a\otimes a\left(B_k\right)_{I_3}\right\}\nonumber\\
    & \quad -\frac{\lambda^4}{\delta^2}\tau\otimes \tau\left(\pi\setminus\{V_1,\ldots,V_k,V\},V_1^{(w_1)},\ldots,V_k^{(w_k)}\right).
    \end{align}
    Let us write 
    \begin{align*}
        B_k=\frac{1}{\delta}\left(a_{i_j}^{(l)}\otimes a_{i_j}^{(r)}-\lambda^2\right)_{j \in [p]}.
    \end{align*}
    Then
    \begin{align*}
        &\tau\otimes \tau\left\{\left(B_k\right)_{I_1}a\otimes a\left(B_k\right)_{I_2}a\otimes a\left(B_k\right)_{I_3}\right\}\\
        &=\frac{1}{\delta^{h}}\sum_{\substack{J_1 \subseteq I_1\\ J_2 \subseteq I_2\\ J_3 \subseteq I_3}}(-1)^{|J_1|+|J_2|+|J_3|}\lambda^{2(|J_1|+|J_2|+|J_3|)}\tau\left(a_{J_1^c}^{(l)}\, a \, a_{J_2^c}^{(l)}\, a \, a_{J_3^c}^{(l)}\right)\tau\left(a_{J_1^c}^{(r)}\, a \, a_{J_2^c}^{(r)}\, a\, a_{J_3^c}^{(r)}\right),
    \end{align*}
    where $h$ is defined as in \eqref{eq: power of delta}. We use Lemma \ref{lemma: contribution of ac1ac2} and compute
    \begin{align*}
        &R(J_1,J_2,J_3) \\
        &:=\tau\left(a_{J_1^c}^{(l)}\, a \, a_{J_2^c}^{(l)}\, a \, a_{J_3^c}^{(l)}\right)\tau\left(a_{J_1^c}^{(r)}\, a \, a_{J_2^c}^{(r)}\, a\, a_{J_3^c}^{(r)}\right)\\
        &=\left(\sigma^2\tau(a_{J_1^c}^{(l)}\, a_{J_3^c}^{(l)})\tau(a_{J_2^c}^{(l)})+\lambda^2\tau(a_{J_1^c}^{(l)}\, a_{J_2^c}^{(l)}\, a_{J_3^c}^{(l)})\right)
        \left(\sigma^2\tau(a_{J_1^c}^{(r)}\, a_{J_3^c}^{(r)})\tau(a_{J_2^c}^{(r)})+\lambda^2\tau(a_{J_1^c}^{(r)}\, a_{J_2^c}^{(r)}\, a_{J_3^c}^{(r)})\right).
    \end{align*}
    After the expansion, we have
    \begin{align*}
        R(J_1,J_2,J_3)=:\sigma^4R_1(J_1,J_2,J_3)+\sigma^2\lambda^2 \left(R_2(J_1,J_2,J_3)+R_3(J_1,J_2,J_3)\right)+\lambda^4R_4(J_1,J_2,J_3),
    \end{align*}
    where we choose $R_2$ to be the term with factor $\tau(a_{J_1^c}^{(l)}\, a_{J_2^c}^{(l)}\, a_{J_3^c}^{(l)})$. For $R_1$, we have
    \begin{align*}
        &\frac{1}{\delta^h}\sum_{\substack{J_1 \subseteq I_1\\ J_2 \subseteq I_2\\ J_3 \subseteq I_3}}(-1)^{|J_1|+|J_2|+|J_3|}\lambda^{2(|J_1|+|J_2|+|J_3|)}R_1(J_1,J_2,J_3)\\
        &=\tau\otimes \tau \left(\left(B_k\right)_{I_1}\left(B_k\right)_{I_3}\right)\tau\otimes \tau\left(\left(B_k\right)_{I_2}\right).
    \end{align*}
    Since $\pi$ is connected, $V$ crosses at least one block $U=\{j,j'\} \in \pi$. Therefore, we assume $j \in I_2$ whose matching symbol $j' \in I_1\cup I_3$. By Lemma \ref{lemma: centering}, we have
    \begin{align*}
        \tau\otimes \tau\left(\left(B_k\right)_{I_2}\right)=0,
    \end{align*}
    and the contribution of $R_1$ is zero. For $R_4$, we have
    \begin{align*}
        &\frac{1}{\delta^h}\sum_{\substack{J_1 \subseteq I_1\\ J_2 \subseteq I_2\\ J_3 \subseteq I_3}}(-1)^{|J_1|+|J_2|+|J_3|}\lambda^{2(|J_1|+|J_2|+|J_3|)}R_4(J_1,J_2,J_3)\\
        &=\tau\otimes \tau \left(\left(B_k\right)_{I_1}\left(B_k\right)_{I_2}\left(B_k\right)_{I_3}\right).
    \end{align*}
    Since
    \begin{align*}
        \tau\otimes \tau \left(\left(B_k\right)_{I_1}\left(B_k\right)_{I_2}\left(B_k\right)_{I_3}\right)=\tau\otimes\tau\left(\pi\setminus\{V_1,\ldots,V_k,V\},V_1^{(w_1)},\ldots,V_k^{(w_k)}\right),
    \end{align*}
    its contribution cancels out with $\frac{\lambda^4}{\delta^2}\tau\otimes \tau\left(\pi\setminus\{V_1,\ldots,V_k,V\},V_1^{(w_1)},\ldots,V_k^{(w_k)}\right)$ in \eqref{eq: value of T}. For $R_2$, let us first denote
    \begin{align*}
        B_{k+1}=B_{k+1}\left(V_1^{(w_1)},\ldots,V_{k}^{(w_k)},V^{(1)}\right)=:\frac{1}{\delta}\left(\overline{a}_{i_j}^{(l)}\otimes \overline{a}_{i_j}^{(r)}-\lambda^2\right)_{j \in [p]}.
    \end{align*}
    We then have
    \begin{align*}
        &\tau\otimes\tau\left(\pi\setminus\{V_1,\ldots,V_k,V\},V_1^{(w_1)},\ldots,V_{k}^{(w_k)},V^{(1)}\right)\\
        &=\frac{1}{\delta^h}\sum_{\substack{J_1 \subseteq I_1\\ J_2 \subseteq I_2\\ J_3 \subseteq I_3}}(-1)^{|J_1|+|J_2|+|J_3|}\lambda^{2(|J_1|+|J_2|+|J_3|)}\tau\left(\overline{a}_{J_1^c}^{(l)}\, \overline{a}_{J_2^c}^{(l)}\, \overline{a}_{J_3^c}^{(l)}\right)\tau\left(\overline{a}_{J_1^c}^{(r)}\, \overline{a}_{J_2^c}^{(r)}\, \overline{a}_{J_3^c}^{(r)}\right).
    \end{align*}
    By definition, the color $w_{k+1}=1$ does not affect the left legs, hence
    \begin{align*}
        \overline{a}_{i_j}^{(l)}=a_{i_j}^{(l)}.
    \end{align*}
    For the right legs and a block $U=\{t_1<t_2\}$, let us divide it into two cases.
    \begin{enumerate}
        \item\label{case 1} If $U$ does not cross $V$, we have
        \begin{align*}
            &\overline{a}_{i_{t_1}}^{(r)}=a_{i_{t_1}}^{(r)};\\
            &\overline{a}_{i_{t_2}}^{(r)}=a_{i_{t_2}}^{(r)}.
        \end{align*}
        \item\label{case 2} Otherwise, either $t_1 \in I_2$ and $t_2 \in I_1\cup I_3$ or $t_1 \in I_1 \cup I_3$ and $t_2 \in I_2$. By the definition of the color $w_{k+1}=1$, we have
        \begin{align*}
            &\overline{a}_{i_{t_1}}^{(r)}=\tilde{a}_{i_{t_1}};\\
            &\overline{a}_{i_{t_2}}^{(r)}=a_{i_{t_1}}.
        \end{align*}
    \end{enumerate}
    Since $(a_k)_{k \in \N}$ and $(\tilde{a}_k)_{k \in \N}$ are free i.i.d copies of $a$, the following holds. First, let $t \in I_2$. Consider $t'$ its matching symbol, that is, $\{t,t'\} \in \pi$ and $i_t=i_{t'} \ne i_j$, for all $j \notin \{t,t'\}$. Recall that, by the definition in \eqref{equation: a_j is either a or tilde a}, we have
    \begin{align*}
        \overline{a}_{i_t}^{(r)},\overline{a}_{i_{t'}}^{(r)} \in \{a_{i_t},\tilde{a}_{i_t}\}.
    \end{align*}
    If $t' \in I_2$, Case \eqref{case 1} implies that the law of $(\overline{a}_{i_t}^{(r)},\overline{a}_{i_{t'}}^{(r)})$ is equal to the law of $({a}_{i_t}^{(r)},{a}_{i_{t'}}^{(r)})$. If $t' \in I_1\cup I_3$, as $a_{i_t},\tilde{a}_{i_t}$ are free i.i.d and free from $(a_k)_{k \ne i_t},(\tilde{a}_k)_{k \ne i_t}$, we have
    \begin{align}\label{eq: equality in distribution t' not in I_2}
        \left((\overline{a}_{i_j}^{(r)})_{j \in I_2\setminus\{t\}},\overline{a}_{i_t}^{(r)}\right)\overset{d}{=}\left((\overline{a}^{(r)}_{i_j})_{j \in I_2\setminus\{t\}},a_{i_t}\right)\overset{d}{=}\left((\overline{a}^{(r)}_{i_j})_{j \in I_2\setminus\{t\}},a_{i_t}^{(r)}\right).
    \end{align}
    Applying the above equalities in distribution for all $t \in I_2$, namely, Case \eqref{case 1} for $t' \in I_2$ and \eqref{eq: equality in distribution t' not in I_2} for $t' \notin I_2$, we get
    \begin{align*}
        \left(\overline{a}_{i_j}^{(r)}\right)_{j \in I_2}\overset{d}{=}\left(a_{i_j}^{(r)}\right)_{j \in I_2}.
    \end{align*}
    By symmetry, we also have
    \begin{align*}
        \left(\overline{a}_{i_j}^{(r)}\right)_{j \in I_1\cup I_3}\overset{d}{=}\left(a_{i_j}^{(r)}\right)_{j \in I_1\cup I_3}.
    \end{align*}
    Additionally, $\left(\overline{a}_{i_j}^{(r)}\right)_{j \in I_2}$ is now free from $\left(\overline{a}_{i_j}^{(r)}\right)_{j \in I_1\cup I_3}$. Indeed, let us show that each variable $\overline{a}_{i_t}^{(r)}$ is free from $\left(\overline{a}_{i_j}^{(r)}\right)_{j \in I_1\cup I_3}$, for all $t \in I_2$. Consider $t'$ its matching symbol. First, if $t' \in I_2$, the fact that $(a_k)_{k \in \N}$ and $(\tilde{a}_k)_{k \in \N}$ are free i.i.d collections implies that $\overline{a}_{i_t}^{(r)}$ (and $\overline{a}_{i_{t'}}^{(r)}$) is free from $\left(\overline{a}_{i_j}^{(r)}\right)_{j \in I_1\cup I_3}$. If $t' \in I_1\cup I_3$, Case \eqref{case 2} implies that $\overline{a}_{i_t}^{(r)}$ is free from $\overline{a}_{i_{t'}}^{(r)}$ and then again $\overline{a}_{i_t}^{(r)}$ is free from $\left(\overline{a}_{i_j}^{(r)}\right)_{j \in I_1\cup I_3}$. 
    
    In summary, by freeness, we have
    \begin{align*}
        \tau\left(\overline{a}_{J_1^c}^{(l)}\, \overline{a}_{J_2^c}^{(l)}\, \overline{a}_{J_3^c}^{(l)}\right)\tau\left(\overline{a}_{J_1^c}^{(r)}\, \overline{a}_{J_2^c}^{(r)}\, \overline{a}_{J_3^c}^{(r)}\right)&=\tau\left(a_{J_1^c}^{(l)}\, a_{J_2^c}^{(l)}\, a_{J_3^c}^{(l)}\right)\tau\left(a_{J_1^c}^{(r)}\, a_{J_3^c}^{(r)}\right)\tau\left(a_{J_2^c}^{(r)}\right)\\
        &=R_2(J_1,J_2,J_3).
    \end{align*}
    Therefore
    \begin{align*}
        &\frac{1}{\delta^h}\sum_{\substack{J_1 \subseteq I_1\\ J_2 \subseteq I_2\\ J_3 \subseteq I_3}}(-1)^{|J_1|+|J_2|+|J_3|}\lambda^{2(|J_1|+|J_2|+|J_3|)}R_2(J_1,J_2,J_3)\\
        &=\tau\otimes \tau\left(\pi\setminus\{V_1,\ldots,V_k,V\},V_1^{(w_1)},\ldots,V_k^{(w_k)},V^{(1)}\right),
    \end{align*}
    Following the same reasoning, it follows that
    \begin{align*}
        &\frac{1}{\delta^h}\sum_{\substack{J_1 \subseteq I_1\\ J_2 \subseteq I_2\\ J_3 \subseteq I_3}}(-1)^{|J_1|+|J_2|+|J_3|}\lambda^{2(|J_1|+|J_2|+|J_3|)}R_3(J_1,J_2,J_3)\\
        &=\tau\otimes \tau\left(\pi\setminus\{V_1,\ldots,V_k,V\},V_1^{(w_1)},\ldots,V_k^{(w_k)},V^{(0)}\right).
    \end{align*}
    Combining the above relations in $T$, we get
    \begin{align*}
        T=\frac{\sigma^2\lambda^2}{\delta^2}\sum_{w=0,1}\tau\otimes \tau\left(\pi\setminus\{V_1,\ldots,V_k,V\},V_1^{(w_1)},\ldots,V_k^{(w_k)},V^{(w)}\right).
    \end{align*}
    It remains to note that 
    \begin{align*}
        \frac{\sigma^2\lambda^2}{\delta^2}=\frac{q}{2},
    \end{align*}
    to finish the proof. 
\end{proof}

\begin{lemma}[Case $\theta_2=1$]\label{lemma: case theta_2=1}
    Let $p \ge 4$ be an even integer, $\pi \in P_2^{\operatorname{con}}(p)$. Let $V_1,\ldots,V_k \in \pi$ be blocks, $w \in \{0,1\}^k$ be a coloring of $V_1,\ldots,V_k$ and
    \begin{align*}
        B_k=B_k\left(V_1^{(w_1)},\ldots,V_k^{(w_k)}\right)=\left(b_{i_j,k}\right)_{j \in [p]}.
    \end{align*}
    Let $V=\{v_1<v_2\} \in \pi\setminus \{V_1,\ldots,V_{k}\}$ and the joint law 
    \begin{align*}
        \left(b_{i_{v_1},k},b_{i_{v_2},k}\right)\overset{d}{=}\left(b_{i_{v_1}}(\theta),b_{i_{v_2}}(\theta)\right).
    \end{align*}
    If $\theta_2=1$, we have
    \begin{align*}
        &\tau\otimes\tau \left(\pi\setminus \{V_1,\ldots,V_{k}\},V_1^{(w_1)},\ldots,V_{k}^{(w_{k})}\right)\\
        &=\frac{q}{2}\tau\otimes\tau \left(\pi\setminus \{V_1,\ldots,V_{k},V\},V_1^{(w_1)},\ldots,V_{k}^{(w_{k})},V^{(1)}\right).
    \end{align*}
\end{lemma}
\begin{proof}
    We must prove that
    \begin{align*}
        T&:=\tau\otimes\tau \left(\pi\setminus \{V_1,\ldots,V_{k}\},V_1^{(w_1)},\ldots,V_{k}^{(w_{k})}\right)\\
        &=\frac{q}{2}\tau\otimes\tau \left(\pi\setminus \{V_1,\ldots,V_{k},V\},V_1^{(w_1)},\ldots,V_{k}^{(w_{k})},V^{(1)}\right).
    \end{align*}
     As $\theta_2=1$, we can write
    \begin{align*}
        &b_{i_{v_1},k}=\frac{1}{\delta}\left(\tilde{a}\otimes a-\lambda^2\right);\\
        &b_{i_{v_2},k}=\frac{1}{\delta}\left(a\otimes a-\lambda^2\right),
    \end{align*}
    where $a:=a_{i_{v_1}}$. We have
    \begin{align}\label{eq: T case theta_2=1}
        T&=\frac{1}{\delta^2}\tau\otimes \tau\left\{\left(B_k\right)_{I_1}\tilde{a}\otimes a\left(B_k\right)_{I_2}a\otimes a\left(B_k\right)_{I_3}\right\}\nonumber \\
    &\quad -\frac{\lambda^4}{\delta^2}\tau\otimes \tau\left(\pi\setminus\{V_1,\ldots,V_k,V\},V_1^{(w_1)},\ldots,V_k^{(w_k)}\right).
    \end{align}
    We write again
    \begin{align*}
        B_k=\frac{1}{\delta}\left(a_{i_j}^{(l)}\otimes a_{i_j}^{(r)}-\lambda^2\right)_{j \in [p]}.
    \end{align*}
    Then
    \begin{align*}
        &\tau\otimes \tau\left\{\left(B_k\right)_{I_1}\tilde{a}\otimes a\left(B_k\right)_{I_2}a\otimes a\left(B_k\right)_{I_3}\right\}\\
        &=\frac{1}{\delta^h}\sum_{\substack{J_1 \subseteq I_1\\ J_2 \subseteq I_2\\ J_3 \subseteq I_3}}(-1)^{|J_1|+|J_2|+|J_3|}\lambda^{2(|J_1|+|J_2|+|J_3|)}\tau\left(a_{J_1^c}^{(l)}\, \tilde{a} \, a_{J_2^c}^{(l)}\, a \, a_{J_3^c}^{(l)}\right)\tau\left(a_{J_1^c}^{(r)}\, a \, a_{J_2^c}^{(r)}\, a\, a_{J_3^c}^{(r)}\right),
    \end{align*}
    where $h$ is defined in \eqref{eq: power of delta}. We use Lemmas \ref{lemma: contribution of ac1ac2} and \ref{lemma: centering} to compute
    \begin{align*}
        R(J_1,J_2,J_3)&:=\tau\left(a_{J_1^c}^{(l)}\, \tilde{a} \, a_{J_2^c}^{(l)}\, a \, a_{J_3^c}^{(l)}\right)\tau\left(a_{J_1^c}^{(r)}\, a \, a_{J_2^c}^{(r)}\, a\, a_{J_3^c}^{(r)}\right)\\
        &=\lambda^2\tau\left(a_{J_1^c}^{(l)}\, a_{J_2^c}^{(l)}\, a_{J_3^c}^{(l)}\right)\left(\sigma^2\tau(a_{J_1^c}^{(r)}\, a_{J_3^c}^{(r)})\tau(a_{J_2^c}^{(r)})+\lambda^2\tau(a_{J_1^c}^{(r)}\, a_{J_2^c}^{(r)}\, a_{J_3^c}^{(r)})\right)
    \end{align*}
    Hence
    \begin{align*}
        R(J_1,J_2,J_3)&=\lambda^2\sigma^2\tau\left(a_{J_1^c}^{(l)}\, a_{J_2^c}^{(l)}\, a_{J_3^c}^{(l)}\right)\tau\left(a_{J_1^c}^{(r)}\, a_{J_3^c}^{(r)}\right)\tau\left(a_{J_2^c}^{(r)}\right) \\
        &\quad+\lambda^4\tau\left(a_{J_1^c}^{(l)}\, a_{J_2^c}^{(l)}\, a_{J_3^c}^{(l)}\right)\tau\left(a_{J_1^c}^{(r)}\, a_{J_2^c}^{(r)}\, a_{J_3^c}^{(r)}\right).
    \end{align*}
    Both terms already appeared in the proof of Lemma \ref{lemma: case theta_1=1}. The second term cancels out with $\frac{\lambda^4}{\delta^2}\tau\otimes \tau\left(\pi\setminus\{V_1,\ldots, V_k,V\}, V_1^{(w_1)},\ldots, V_k^{(w_k)}\right)$ in \eqref{eq: T case theta_2=1}, and the contribution of the first is equal to
    \begin{align*}
        &\frac{1}{\delta^h}\sum_{\substack{J_1 \subseteq I_1\\ J_2 \subseteq I_2\\ J_3 \subseteq I_3}}(-1)^{|J_1|+|J_2|+|J_3|}\lambda^{2(|J_1|+|J_2|+|J_3|)}\tau\left(a_{J_1^c}^{(l)}\, a_{J_2^c}^{(l)}\, a_{J_3^c}^{(l)}\right)\tau\left(a_{J_1^c}^{(r)}\, a_{J_3^c}^{(r)}\right)\tau\left(a_{J_2^c}^{(r)}\right)\\
        &=\tau\otimes \tau\left(\pi\setminus\{V_1,\ldots,V_k,V\},V_1^{(w_1)},\ldots,V_k^{(w_k)},V^{(1)}\right).
    \end{align*}
    We then have
    \begin{align*}
        T=\frac{\sigma^2\lambda^2}{\delta^2}\tau\otimes \tau\left(\pi\setminus\{V_1,\ldots,V_k,V\},V_1^{(w_1)},\ldots,V_k^{(w_k)},V^{(1)}\right),
    \end{align*}
    and the result follows by the definition of $q$.
\end{proof}

The case $\theta_3=1$ is symmetric to the case $\theta_2=1$ in Lemma \ref{lemma: case theta_2=1}, and we omit the proof.
\begin{lemma}[Case $\theta_3=1$]\label{lemma: case theta_3=1}
    Let $p \ge 4$ be an even integer, $\pi \in P_2^{\operatorname{con}}(p)$. Let $V_1,\ldots,V_k \in \pi$ be blocks, $w \in \{0,1\}^k$ be a coloring of $V_1,\ldots,V_k$ and
    \begin{align*}
        B_k=B_k\left(V_1^{(w_1)},\ldots,V_k^{(w_k)}\right)=\left(b_{i_j,k}\right)_{j \in [p]}.
    \end{align*}
    Let $V=\{v_1<v_2\} \in \pi\setminus \{V_1,\ldots,V_{k}\}$ and the joint law 
    \begin{align*}
        \left(b_{i_{v_1},k},b_{i_{v_2},k}\right)\overset{d}{=}\left(b_{i_{v_1}}(\theta),b_{i_{v_2}}(\theta)\right).
    \end{align*}
    If $\theta_3=1$, we have
    \begin{align*}
        &\tau\otimes\tau \left(\pi\setminus \{V_1,\ldots,V_{k}\},V_1^{(w_1)},\ldots,V_{k}^{(w_{k})}\right)\\
        &=\frac{q}{2}\tau\otimes\tau \left(\pi\setminus \{V_1,\ldots,V_{k},V\},V_1^{(w_1)},\ldots,V_{k}^{(w_{k})},V^{(0)}\right).
    \end{align*}
\end{lemma}

Finally, the case $\theta_4=1$ has a null contribution.
\begin{lemma}[Case $\theta_4=1$]\label{lemma: case theta_4=1}
    Let $p \ge 4$ be an even integer, $\pi \in P_2^{\operatorname{con}}(p)$. Let $V_1,\ldots,V_k \in \pi$ be blocks, $w \in \{0,1\}^k$ be a coloring of $V_1,\ldots,V_k$ and
    \begin{align*}
        B_k=B_k\left(V_1^{(w_1)},\ldots,V_k^{(w_k)}\right)=\left(b_{i_j,k}\right)_{j \in [p]}.
    \end{align*}
    Let $V=\{v_1<v_2\} \in \pi\setminus \{V_1,\ldots,V_{k}\}$ and the joint law 
    \begin{align*}
        \left(b_{i_{v_1},k},b_{i_{v_2},k}\right)\overset{d}{=}\left(b_{i_{v_1}}(\theta),b_{i_{v_2}}(\theta)\right).
    \end{align*}
    If $\theta_4=1$, we have
    \begin{align*}
        &\tau\otimes\tau \left(\pi\setminus \{V_1,\ldots,V_{k}\},V_1^{(w_1)},\ldots,V_{k}^{(w_{k})}\right)=0.
    \end{align*}
\end{lemma}
\begin{proof}
    Similarly to the previous proofs, we must show that
    \begin{align*}
        T:=\tau\otimes\tau \left(\pi\setminus \{V_1,\ldots,V_{k}\},V_1^{(w_1)},\ldots,V_{k}^{(w_{k})}\right) =0.
    \end{align*}
     As $\theta_4=1$, we can write
    \begin{align*}
        &b_{i_{v_1},k}=\frac{1}{\delta}\left(\tilde{a}\otimes \tilde{a}-\lambda^2\right);\\
        &b_{i_{v_2},k}=\frac{1}{\delta}\left(a\otimes a-\lambda^2\right),
    \end{align*}
    where $a:=a_{i_{v_1}}$. First note that
    \begin{align*}
        T&=\frac{1}{\delta^2}\tau\otimes \tau\left\{\left(B_k\right)_{I_1}\tilde{a}\otimes \tilde{a}\left(B_k\right)_{I_2}a\otimes a\left(B_k\right)_{I_3}\right\}\\
    &\quad-\frac{\lambda^4}{\delta^2}\tau\otimes \tau\left(\pi\setminus\{V_1,\ldots,V_k,V\},V_1^{(w_1)},\ldots,V_k^{(w_k)}\right).
    \end{align*}
    Applying Lemma \ref{lemma: centering} to both tensors, we get that
    \begin{align*}
        \tau\otimes \tau\left\{\left(B_k\right)_{I_1}\tilde{a}\otimes \tilde{a}\left(B_k\right)_{I_2}a\otimes a\left(B_k\right)_{I_3}\right\}=\lambda^4\tau\otimes \tau\left\{\left(B_k\right)_{I_1}\left(B_k\right)_{I_2}\left(B_k\right)_{I_3}\right\}.
    \end{align*}
    The result follows by recalling that
    \begin{align*}
        \tau\otimes \tau\left\{\left(B_k\right)_{I_1}\left(B_k\right)_{I_2}\left(B_k\right)_{I_3}\right\}&=\tau\otimes\tau \left(\left(B_k\right)_{[p]\setminus\left(V_1\cup \cdots \cup V_k \cup V\right)}\right)\\
        &=\tau\otimes \tau\left(\pi\setminus\{V_1,\ldots,V_k,V\},V_1^{(w_1)},\ldots,V_k^{(w_k)}\right).
    \end{align*}
\end{proof}

\begin{proof}[Proof of Proposition \ref{proposition: contribution bipartite connected part}]
    Let $\pi \notin P_2^{\operatorname{bicon}}(p)$, and let $2l^*+1$ be the length of the smallest odd cycle in the intersection graph of $\pi$. We denote $V_1,\ldots, V_{2l^*+1} \in \pi$ (any) ordered sequence of blocks in this smallest odd cycle, where $V_{j}$ crosses $V_{j+1}$ for each $j=1,\ldots,2l^*+1$ and $V_{2l^*+2}:=V_1$. We first remove the first block $V_1$, and by Lemma \ref{lemma: case theta_1=1} we get that
    \begin{align*}
        \tau\otimes\tau(\pi)=\frac{q}{2}\sum_{w=0,1}\tau\otimes\tau\left(\pi\setminus V_1,V_1^{(w)}\right).
    \end{align*}
    However, as it is the first block we remove, the distribution of $B_0=(b_{i_j})_{j \in [p]}$ is symmetric in both legs. In particular, we have
    \begin{align*}
        B_1(V_1^{(0)})\overset{d}{=}B_1(V_1^{(1)}),
    \end{align*}
    and
    \begin{align*}
        \tau\otimes\tau\left(\pi\setminus V_1,V_1^{(0)}\right)=\tau\otimes\tau\left(\pi\setminus V_1,V_1^{(1)}\right).
    \end{align*}
    Hence
    \begin{align*}
        \tau\otimes\tau(\pi)=q\tau\otimes\tau\left(\pi\setminus V_1,V_1^{(0)}\right).
    \end{align*}
    We then fix the color of $V_1$ to be $w_1=0$. We aim to show by induction that the colors $w_1,w_2,\ldots,w_{2l^*}$ of $V_1,\ldots,V_{2l^*}$ are deterministic and alternating, 
    \begin{align*}
        w=(w_1,\ldots,w_{2l^*})=(0,1,0,\ldots,1),
    \end{align*}
    and so is $\theta^{(k+1)}:=\theta\left(V_k,V_1^{(w_1)},\ldots,V_{k}^{(w_{k})}\right)$ for $1 \le k \le 2l^*$ in the sense that
    \begin{align*}
        \left(\theta^{(2)}_2,\theta^{(3)}_3,\theta^{(4)}_2,\ldots,\theta^{(2l^*)}_2\right)=(1,1,1,\ldots,1).
    \end{align*}
    First, as $V_2$ crosses $V_1$ and $w_1=0$, we have $\theta^{(2)}_2=1$. By Lemma \ref{lemma: case theta_2=1}, we have $w_2=1$. Suppose then that the result holds for some $1 < k < 2l^*$. Let us prove that it also holds for $k+1$. To simplify the reading, let us divide into two cases whether $k$ is even or odd.

    \textit{Case $k$ is even.} The induction hypothesis implies that
    \begin{align*}
        &(w_2,w_3,\ldots,w_k)=(1,0,\ldots,1);\\
        &\left(\theta^{(2)}_2,\theta^{(3)}_3,\theta^{(4)}_2,\ldots,\theta^{(k)}_2\right)=(1,1,1,\ldots,1).
    \end{align*}
    Since $V_{k+1}$ crosses $V_k$, we automatically have that $\theta^{(k+1)}_1=0$. Moreover, as $w_k=1$, the definition of $\theta$ implies that
    \begin{align*}
        \theta^{(k+1)}_2=0.
    \end{align*}
    Hence either $\theta^{(k+1)}_4=1$ or $\theta^{(k+1)}_3=1$. Now consider all blocks $V_{j_t}$ that might potentially cross $V_{k+1}$, for $j_t\le k$. Since $2l^*+1$ is the length of the smallest odd cycle, the cycles $(V_{j_t},\ldots, V_{k}, V_{k+1})$ are of even length. By a parity check, $j_t$ must be even, and hence
    \begin{align*}
        w_{j_t}=1.
    \end{align*}
    This implies that all crossing blocks $V_{j_t}$ of $V_{k+1}$ for $j_t \le k$ have the same color $w_{j_t}=1$. In particular, we have
    \begin{align*}
        \theta^{(k+1)}_{3}=1.
    \end{align*}
    We then apply Lemma \ref{lemma: case theta_3=1} to get that $w_{k+1}=0$.

    \textit{Case $k$ is odd.} This follows similarly to the even case. We first have that $\theta^{(k+1)}_1=0$ and as $w_{k}=0$, we have
    \begin{align*}
        \theta^{(k+1)}_3=0.
    \end{align*}
    By the parity check on the even cycles $V_{k+1}$ might belong to, all crossing blocks $V_{j_t}$ of $V_{k+1}$ for $j_t \le k$ have the same color $w_{j_t}=0$. Hence
    \begin{align*}
        \theta^{(k+1)}_2=1,
    \end{align*}
    and Lemma \ref{lemma: case theta_2=1} implies that $w_{k+1}=1$. This finishes the induction.

    To conclude the proof, the block $V_{2l^*+1}$ crosses both $V_1$ of color $w_1=0$ and $V_{2l^*}$ of color $w_{2l^*}=1$. Hence
    \begin{align*}
        \theta^{(2l^*+1)}_4=1.
    \end{align*}
    Lemma \ref{lemma: case theta_4=1} applied to $V=V_{2l^*+1}$ implies that
    \begin{align*}
        \tau\otimes\tau\left(\pi\setminus \{V_1,\ldots,V_{2l^*}\},V_1^{(0)},\ldots,V_{2l^*}^{(1)}\right)=0. 
    \end{align*}
    Applying Lemma \ref{lemma: case theta_2=1} for blocks $V=V_k$ when $k$ is even and Lemma \ref{lemma: case theta_3=1} for $V=V_k$ when $k$ is odd, we get that
    \begin{align*}
        \tau\otimes \tau(\pi)&=q\tau\otimes \tau\left(\pi\setminus V_1,V_1^{(0)}\right)\\
        &=q\cdot \frac{q}{2}\tau\otimes \tau\left(\pi\setminus\{V_1,V_2\},V_1^{(0)},V_2^{(1)}\right)\\
        &=\cdots\\
        &=q\left(\frac{q}{2}\right)^{2l^*-1}\tau\otimes \tau\left(\pi\setminus\{V_1,\ldots,V_{2l^*}\},V_1^{(0)},\ldots,V_{2l^*}^{(1)}\right)\\
        &=0.
    \end{align*}
    Therefore, for all $\pi \notin P_2^{\operatorname{bicon}}(p)$, we have $\tau\otimes \tau(\pi)=0$. This proves the first part of the proposition.

    Assume now that $\pi \in P_2^{\operatorname{bicon}}(p)$. Let $V_1,\ldots,V_{p/2}$ be an ordering of the blocks of $\pi$ such that $V_k$ crosses at least one $V_j$ for $j<k$ and all $k$. This can always be done as $\pi$ is connected. Let $\mathcal{V}_1,\mathcal{V}_2$ be the bipartite sets of vertices of the intersection graph of $\pi$,
    \begin{align*}
        &\mathcal{V}_1=\{V_1,V_{j_1},\ldots,V_{j_m}\};\\
        &\mathcal{V}_2=\{V_2,V_{h_1},\ldots,V_{h_n}\}.
    \end{align*}
    Since $\pi$ is bipartite, if $V_k \in \mathcal{V}_1$, all crossing blocks $V_j$ of $V_k$ belong to $\mathcal{V}_2$, and similarly if $V_k \in \mathcal{V}_2$. We assume again that the color of $V_1$ is $w_1=0$, since
    \begin{align*}
        \tau\otimes \tau(\pi)=q \tau\otimes\tau \left(\pi\setminus V_1,V_1^{(0)}\right).
    \end{align*}
    We will prove that both the color $w_k$ and the binary vector $\theta^{(k)}$ of $V_k$ depend only on which bipartite set $V_k$ belongs to, namely, for all $k \ge 2$, the following hold.
    \begin{enumerate}
        \item If $V_k \in \mathcal{V}_1$, then $w_k=0$ and $\theta^{(k)}_{3}=1$;
        \item Otherwise, $V_k \in \mathcal{V}_2$, $w_k=1$ and $\theta^{(k)}_2=1$.
    \end{enumerate}
    Indeed, note first that $V_2 \in \mathcal{V}_2$, $\theta^{(2)}_2=1$ as it crosses $V_1$ and $w_1=0$. Then, Lemma \ref{lemma: case theta_2=1} implies that $w_2=1$. Assume then that the result holds for some $l$ and for all $1 \le k \le l$. Let us prove that it holds for $V_{l+1}$ as well. Indeed, assume without loss of generality that $V_{l+1} \in \mathcal{V}_1$. Then $V_{l+1}$ only crosses blocks $V_{j_t}$ such that $V_{j_t} \in \mathcal{V}_2$, for $j_t \le l$. Since all colors $w_{j_t}=1$, we deduce that
    \begin{align*}
        \theta^{(l+1)}_3=1.
    \end{align*}
    Then, Lemma \ref{lemma: case theta_3=1} implies that $w_{l+1}=0$ and induction is proved. We then apply Lemma \ref{lemma: case theta_1=1} for $V_1$, Lemma \ref{lemma: case theta_2=1} for all blocks $V_k \in \mathcal{V}_2$ and Lemma \ref{lemma: case theta_3=1} for all blocks $V_k \in \mathcal{V}_1$ for $k \ge 2$, so that
    \begin{align*}
        \tau\otimes \tau(\pi)&=q\tau\otimes \tau\left(\pi\setminus V_1,V_1^{(0)}\right)\\
        &=q\cdot \frac{q}{2}\tau\otimes \tau\left(\pi\setminus\{V_1,V_2\},V_1^{(0)},V_2^{(1)}\right)\\
        &=\cdots\\
        &=q\left(\frac{q}{2}\right)^{p/2-1}.
    \end{align*}
    It follows then that
    \begin{align*}
        \tau\otimes \tau(\pi)=q\left(\frac{q}{2}\right)^{p/2-1}=2\left(\frac{q}{2}\right)^{p/2},
    \end{align*}
    and the result is proved.
\end{proof}

\subsection{Proof of Theorem~\ref{th: main theorem}}
We are now ready to prove the main theorem. 
Recall that 
\begin{align*}
    \tau'({\bf{S}}^{p})=\sum_{\pi \in P_2(p)}\tau\otimes \tau(\pi)= \sum_{\hat{\pi},(\pi_T)_{T \in \hat{\pi}}}\prod_{T \in \hat{\pi}}\tau\otimes \tau(\pi_T), 
\end{align*}
where the second summation runs over $(\hat{\pi},(\pi_T)_{T \in \hat{\pi}})\in \Phi(P_2(p))$, using the bijection $\Phi$ defined in \eqref{equation: mapping pi connected components}.
Note that if $|T|=2$, then by Corollary~\ref{corollary: contribution of noncrossing partitions} we have $\tau\otimes \tau(\pi_T)=1$. 
Therefore, we deduce that 
\begin{align*}
    \tau'({\bf{S}}^{p})= \sum_{\hat{\pi},(\pi_T)_{T \in \hat{\pi}}} \prod_{\underset{|T|\geq 4}{T \in \hat{\pi}}}\tau\otimes \tau(\pi_T). 
\end{align*}
Using Proposition~\ref{proposition: contribution bipartite connected part}, we get
\begin{align*}
    \tau'({\bf{S}}^{p})= \sum_{\hat{\pi}\in NC(p)}\sum_{(\pi_T)_{T \in \hat{\pi}}}\prod_{\underset{|T|\geq 4}{T \in \hat{\pi}}}2 \Big(\frac{q}{2}\Big)^{\frac{|T|}{2}},
\end{align*} 
where the second summation runs over bipartite connected pair partitions $\pi_T$, for $T \in \hat{\pi}$.
Finally, note that the number of size-two blocks is precisely $\ncrblocks(\pi)$ and thus
\begin{align*}
    &|\{T \in \hat{\pi}:|T| \ge 4\}|=\ccblocks(\pi)-\ncrblocks(\pi).
\end{align*}
Since     
$$
\sum_{\substack{T \in \hat{\pi}\\ |T| \ge 4}}\frac{|T|}{2}=\crblocks(\pi),
$$
using again the bijection $\Phi$, we deduce that 
\begin{align*}
    \tau'({\bf{S}}^{p})= \sum_{\pi \in P_2^{\operatorname{bi}}(p)} 2^{\ccblocks(\pi)-\ncrblocks(\pi)} \Big(\frac{q}{2}\Big)^{\crblocks(\pi)}. 
\end{align*} 
This finishes the proof in view of Proposition~\ref{prop: moments of mu_q} and of the fact that
\begin{align*}
    \ncrblocks(\pi)+\crblocks(\pi)=|\pi|=\frac{p}{2}.
\end{align*}

\end{document}